%% file: main.tex
\pdfoutput=1
\documentclass[letterpaper,11pt]{article}
\usepackage[margin=1in]{geometry}

\usepackage[T1]{fontenc}
\usepackage{graphicx}
\usepackage{amsmath,amssymb,amsfonts}
\usepackage{amsthm}
\usepackage{mathrsfs}
\usepackage[title]{appendix}
\usepackage{xcolor}
\usepackage{booktabs}
\usepackage{enumitem}
\usepackage{microtype}
\usepackage{mathtools}
\usepackage{tikz}

\usepackage{hyperref}
\usepackage[capitalize,nameinlink,noabbrev]{cleveref}

\newcommand{\blue}[1]{{\color{blue}#1}}
\newcommand{\linkToPdf}[1]{\href{#1}{\blue{(pdf)}}}

\hypersetup{colorlinks, hypertexnames=false, pageanchor=true,
            linkcolor=blue, citecolor={green!50!black}, urlcolor=cyan,
            pdftitle={Benign Degeneracy in Exact Boolean Moment-SOS
                      Relaxations},
            pdfauthor={Shucheng Kang}
            }
\mathtoolsset{centercolon}
\numberwithin{equation}{section}

\makeatletter
\AddToHook{cmd/appendix/before}{%
  \def\cref@section@alias{appendix}%
  \def\cref@subsection@alias{appsubsec}}
\makeatother

\input{macros.tex}

\title{Benign Degeneracy for First-Order Methods \\ in Exact Boolean Moment--SOS Relaxations}

\author{
Shucheng Kang\\[3pt]
\normalsize School of Engineering and Applied Sciences, Harvard University\\
\normalsize \texttt{skang1@g.harvard.edu}
}
\date{}

\begin{document}

\maketitle
\thispagestyle{empty}

\begin{abstract}\normalsize
\input{sections/abstract.tex}
\end{abstract}

\newpage
\thispagestyle{empty}
\setcounter{tocdepth}{1} 
\tableofcontents

\newpage
\setcounter{page}{1}

\input{sections/introduction.tex}
\input{sections/degeneracy.tex}
\input{sections/rigidity.tex}
\input{sections/boolean.tex}

\input{sections/algorithms.tex}
\input{sections/genericity.tex}

\input{sections/discussion.tex}

\clearpage
\appendix
\input{sections/app_related.tex}
\input{sections/app_formulation.tex}
\input{sections/app_spectral.tex}
\input{sections/app_rigidity_proofs.tex}
\input{sections/app_modulus.tex}
\input{sections/app_algorithms.tex}
\input{sections/app_genericity.tex}
\input{sections/app_experiments.tex}

\clearpage

\section*{Acknowledgments}
The author thanks Heng Yang for helpful discussions. This project is
partly funded by the NSF CAREER Award 2543352.

\input{sections/ai_disclosure.tex}

\bibliographystyle{alphaurl}
\bibliography{../../references/refs,../../references/myRefs}

\end{document}

%% file: macros.tex
\newcommand{\R}{\mathbb R}
\newcommand{\Sym}{\mathbb S}          
\newcommand{\cA}{\mathcal A}
\newcommand{\cR}{\mathcal R}
\newcommand{\cN}{\mathcal N}
\newcommand{\cH}{\mathcal H}
\newcommand{\cT}{\mathcal T}
\newcommand{\cI}{\mathcal I}
\newcommand{\cP}{\mathcal P}
\newcommand{\cQ}{\mathcal Q}
\newcommand{\cC}{\mathcal C}
\newcommand{\cG}{\mathcal G}          
\newcommand{\cK}{\mathcal K}          
\newcommand{\cX}{\mathcal X}          
\newcommand{\Ggen}{\mathcal G_{\mathrm{gen}}} 

\newcommand{\Sopt}{\mathcal S_\star}  
\newcommand{\Zopt}{\mathcal Z_\star}  
\newcommand{\Wopt}{\mathcal W_\star}  
\newcommand{\Ssc}{S_{\mathrm{sc}}}    
\newcommand{\Xsc}{X_{\mathrm{sc}}}    
\newcommand{\Xopt}{\mathcal X_\star}  
\newcommand{\cM}{\mathcal M}          
\newcommand{\cL}{\mathcal L}          
\newcommand{\cD}{\mathcal D}          
\newcommand{\cE}{\mathcal E}          
\newcommand{\Tan}{T}                  
\DeclareMathOperator{\rb}{rb}

\newcommand{\Prow}{\mathsf P}         
\newcommand{\Pnull}{\mathsf P^\perp}  
\newcommand{\Jext}{\mathscr J}        
\newcommand{\Id}{\mathrm{Id}}

\DeclareMathOperator{\tr}{tr}
\DeclareMathOperator{\rank}{rank}
\DeclareMathOperator{\ran}{ran}
\DeclareMathOperator{\spanof}{span}
\DeclareMathOperator{\aff}{aff}
\DeclareMathOperator{\ri}{ri}
\DeclareMathOperator{\cl}{cl}
\DeclareMathOperator{\dist}{dist}
\DeclareMathOperator{\Fix}{Fix}
\DeclareMathOperator{\diag}{diag}
\DeclareMathOperator{\OPT}{OPT}
\DeclareMathOperator*{\argmin}{argmin}

\newcommand{\ip}[2]{\langle #1,#2\rangle}
\newcommand{\normF}[1]{\lVert #1\rVert_F}
\newcommand{\normtwo}[1]{\lVert #1\rVert_2}
\newcommand{\normop}[1]{\lVert #1\rVert_{\mathrm{op}}}
\newcommand{\normc}[1]{\lVert #1\rVert_{\mathrm{coef}}}

\newcommand{\eps}{\varepsilon}

\newcommand{\rpar}{r^{\parallel}}
\newcommand{\rper}{r^{\perp}}

\theoremstyle{plain}
\newtheorem{theorem}{Theorem}[section]
\newtheorem{proposition}[theorem]{Proposition}
\newtheorem{lemma}[theorem]{Lemma}
\newtheorem{corollary}[theorem]{Corollary}
\newtheorem{assumption}[theorem]{Assumption}
\theoremstyle{definition}

\theoremstyle{remark}
\newtheorem{remark}[theorem]{Remark}
\newtheorem{example}[theorem]{Example}

\crefname{assumption}{Assumption}{Assumptions}
\Crefname{assumption}{Assumption}{Assumptions}

\crefformat{section}{\S#2#1#3}
\crefformat{subsection}{\S#2#1#3}
\crefformat{subsubsection}{\S#2#1#3}
\Crefformat{section}{\S#2#1#3}
\Crefformat{subsection}{\S#2#1#3}
\Crefformat{subsubsection}{\S#2#1#3}
\crefmultiformat{section}{\S#2#1#3}{ and \S#2#1#3}{, \S#2#1#3}{, and \S#2#1#3}
\crefrangeformat{section}{\S#3#1#4--\S#5#2#6}
\Crefmultiformat{section}{\S#2#1#3}{ and \S#2#1#3}{, \S#2#1#3}{, and \S#2#1#3}
\Crefrangeformat{section}{\S#3#1#4--\S#5#2#6}
\crefmultiformat{subsection}{\S#2#1#3}{ and \S#2#1#3}{, \S#2#1#3}{, and \S#2#1#3}
\crefrangeformat{subsection}{\S#3#1#4--\S#5#2#6}
\Crefmultiformat{subsection}{\S#2#1#3}{ and \S#2#1#3}{, \S#2#1#3}{, and \S#2#1#3}
\Crefrangeformat{subsection}{\S#3#1#4--\S#5#2#6}
\crefname{appendix}{Appendix}{Appendices}
\Crefname{appendix}{Appendix}{Appendices}
\crefname{appsubsec}{Appendix}{Appendices}
\Crefname{appsubsec}{Appendix}{Appendices}

%% file: sections/abstract.tex
Pseudo-Boolean optimization---minimizing a polynomial over
$\{\pm1\}^n$---includes MAX-CUT, binary quadratic optimization, and Ising
models. The moment--sum-of-squares (Moment--SOS) hierarchy provides
certified semidefinite relaxations that are often empirically exact at low
order. We ask: once a relaxation is exact, how difficult is the resulting
semidefinite program (SDP) to solve with first-order methods?

Exactness does not imply regularity. At every order $d\ge2$, a unique
rank-one moment optimum is primal degenerate; moreover, if one optimal SOS
certificate is strictly complementary, then the certificate set is
positive-dimensional and contains singular boundary points. Thus first-order
methods may converge to such singular points, around which pointwise
analyses that assume regularity at the limit cannot establish a local
linear convergence rate.

We show that, for exact Boolean Moment--SOS relaxations in their canonical
quotient form, this degeneracy is nevertheless benign. We develop an
abstract mechanism, which we call \emph{residual-field rigidity}, that
propagates regularity from the relative interior of the optimal set to its
singular boundary. It requires neither the primal nor the dual optimal
solution to be unique, and it is equivalent to the first-order stall
directions of the residual coinciding with the tangent directions of the
optimal set. For Boolean
Moment--SOS, unique rank-one exactness and the existence of one strictly
complementary certificate suffice to trigger this mechanism, yielding a
uniform error bound around the entire optimal set. Consequently,
ADMM, sGS--ADMM, and PDHG converge eventually linearly
from every initialization, even when their limiting certificate is
singular. Finally, outside a
lower-dimensional semialgebraic exceptional set of Boolean objectives, these
hypotheses hold simultaneously at every uniquely rank-one exact order.

%% file: sections/introduction.tex
\section{Introduction}
\label{sec:intro}

\paragraph{Problem statement.}
Consider the pseudo-Boolean optimization (PBO): given a real
polynomial $f$ in $n$ variables, compute
\begin{equation}\label{eq:pbo}
  \OPT(f):=\min_{x\in\{\pm1\}^n} f(x).
\end{equation}
Since $x_i^2=1$ on the cube, we always represent $f$ by its unique
square-free (multilinear) polynomial. Problem~\eqref{eq:pbo} is a broad
modeling language for discrete optimization~\cite{boros02dam-pbo}. Already
at degree two it contains MAX-CUT, unconstrained binary quadratic
optimization (QUBO)~\cite{kochenberger14joco-ubqp}, and ground states of
Ising spin glasses, which are NP-hard to
compute~\cite{barahona82jpa-ising}; higher degrees cover Boolean
constraint satisfaction. Semidefinite programming (SDP) relaxations
have been central to
this problem class since the Goemans--Williamson approximation algorithm
for MAX-CUT~\cite{goemans95jacm-maxcut}.

\paragraph{The Moment--SOS hierarchy.}
The canonical family of semidefinite relaxations for \eqref{eq:pbo} is the
Moment--SOS hierarchy~\cite{lasserre01siopt-global,parrilo03mp-semidefinite,laurent03mor-comparison}.
At order $d$, a true probability distribution on the cube is replaced by a
degree-$2d$ pseudo-moment sequence, represented by a positive-semidefinite
(PSD) moment matrix; the dual PSD variable is a sum-of-squares (SOS) Gram
certificate. In the language of theoretical computer science, this is the degree-$2d$ pseudoexpectation / SOS proof-system
view~\cite{barak14icm-sosproof,moitra20chapter-sostcs}. The order-$d$ value
$p_d(f)$ is defined at every \emph{admissible} order, $2d\ge\deg f$, and
satisfies
\begin{equation}\label{eq:hierarchy}
  p_d(f)\le \OPT(f),
  \qquad p_{d_{\min}}(f)\le p_{d_{\min}+1}(f)\le\cdots,
  \qquad d_{\min}(f):=\max\{1,\lceil\deg f/2\rceil\},
\end{equation}
and the hierarchy is exact at order $n$ on the finite
cube~\cite{laurent03mor-comparison} (see also \cref{prop:eventual}). We call the
relaxation \emph{exact} at order $d$ if $p_d(f)=\OPT(f)$, and \emph{uniquely
rank-one exact} if in addition its moment SDP has the unique optimizer
$b_d(x_\star)b_d(x_\star)^\top$, where $x_\star$ is the Boolean minimizer and
$b_d$ is the monomial vector, as elaborated in \cref{sec:degeneracy}.

\paragraph{Worst-case looseness versus empirical low-order exactness.}
How large must $d$ be? Worst-case theory is discouraging: the cut polytope
of the complete graph needs relaxation order
$\lceil n/2\rceil$~\cite{laurent03mor-cutpolytope}, knapsack-type
instances carry $\Omega(n)$ SOS-degree lower
bounds~\cite{grigoriev01cc-knapsack}, and random $k$-CSPs keep large gaps
even at order $\Omega(n)$~\cite{schoenebeck08focs-lasserre}. Computational
practice tells a different story. On unconstrained binary
quadratic optimization, order two often delivers the exact solution or a
very good approximation~\cite{habibi25jgo-lasserre-ubqp}. BQP relaxations
are a standard testbed for scalable low-rank
solvers~\cite{yang23mp-stride,hou26or-rnnal,hou25arxiv-rinnalplus,wang25arxiv-dual-riemannian-admm-lowrank-sdp}.
The second-order Boolean relaxation, reported empirically exact on random
QUBO, has been solved at $n=120$ variables, a moment matrix of size
$7{,}261$ with about $17.9$ million equality
constraints~\cite{wang25jsc-manisdp}, and certifiably exact low-order
relaxations are routine in robotics and
perception~\cite{rosen19ijrr-sesync,briales18cvpr-certifiably,yang22pami-certifiably-plain,kang24wafr-strom}.
Most hierarchy theory asks how large $d$ must be before exactness. The
empirical picture motivates a complementary question:
\begin{center}
\emph{Once a low-order Moment--SOS relaxation is exact,\\
how hard is the resulting semidefinite program to solve?}
\end{center}

\paragraph{Why first-order methods.}
For SDPs of small and medium size, interior-point methods are the solver
of choice. They need few iterations~\cite{deklerk02book-ipm} and reach
high accuracy. Under strict complementarity they even converge locally
superlinearly~\cite{alizadeh98siopt-aho,luo98anp-superlinear-convergence-path-following-sdp}.
But each conventional interior-point iteration forms and factors a dense
matrix whose dimension is the number of equality constraints. First-order
splitting methods target the large-scale regime where this matrix can no
longer even be formed~\cite{wen10mpc-admmsdp,odonoghue16jota-scs}. Boolean
moment relaxations sit squarely in that regime, with a moderate PSD block
(the moment matrix of size $7{,}261$ above) and equality constraints
quadratic in its size (the $17.9$ million above). A
splitting iteration needs only one eigendecomposition of the moment
matrix, plus an affine projection through a factorization computed once;
its local convergence rate is therefore the relevant measure of hardness.

\paragraph{Exact but degenerate.}
Can a splitting
method---ADMM~\cite{gabay76cma-dual-algorithm-nonlinear-variation}, its
symmetric Gauss--Seidel variant sGS--ADMM~\cite{chen17mp-sgsadmm}, or
PDHG~\cite{chambolle11jmiv-pdhg}---converge locally linearly on
the exact relaxation? These methods are sublinear in the worst
case~\cite{davis17bookchapter-convergence,zamani24mp-exact}, and locally
linear under regularity conditions at the limit point they
select~\cite{han18mor-linear,kang25arxiv-admmsdp-linearconv,jiang26arxiv-linearconv-pdhg-sdp,ding26arxiv-linearconv-fom-cp,li26arxiv-gpu-conic-qp-linearconv-sc}.
Unfortunately, exact Boolean relaxations sit in a known harder class:
\emph{degenerate} SDPs. Suppose $f$ has a unique Boolean minimizer $x_\star$
and order $d$ is uniquely rank-one exact, so the moment optimum is the
single rank-one matrix $X_\star=b_d(x_\star)b_d(x_\star)^\top$. Then,
this optimum is primal degenerate whenever $d\ge2$, in the standard SDP
sense~\cite{alizadeh9mp-complementarity-nondegeneracy}
(\cref{sec:degeneracy}). Moreover, once one SOS certificate is strictly
complementary---the generic case (\cref{sec:genericity})---the optimal Gram
matrices form a positive-dimensional compact convex set $\Sopt$ whose
relative boundary consists of certificates that lose rank. A first-order
method is free to converge to such a singular boundary certificate, where
the pointwise regularity toolbox breaks down and no local linear rate can be
concluded.

\begin{figure}[t]
\centering
\begin{minipage}[t]{0.32\textwidth}\centering
 \includegraphics[width=\linewidth]{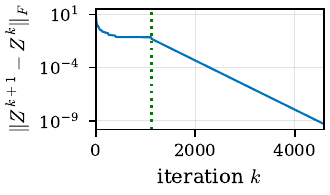}\\[1pt]
 {\small (a) cold start, non-SC limit}
\end{minipage}\hfill
\begin{minipage}[t]{0.32\textwidth}\centering
 \includegraphics[width=\linewidth]{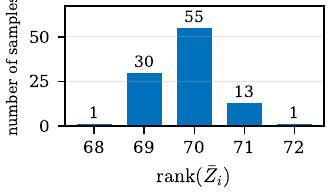}\\[1pt]
 {\small (b) ranks of the endpoints $\bar{Z}_i$}
\end{minipage}\hfill
\begin{minipage}[t]{0.32\textwidth}\centering
 \includegraphics[width=\linewidth]{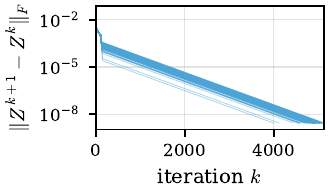}\\[1pt]
 {\small (c) $100$ near-singular warm starts}
\end{minipage}
\caption{Benign degeneracy on second-order moment relaxations of two BQP
instances with $n=20$. (a) Exact ADMM ($\sigma=100$, all-zeros start) on
the standard relaxation ($N=231$) of the instance
of~\cite{kang25arxiv-admmsdp-linearconv}: the limiting pair is \emph{not}
strictly complementary, yet the residual enters a sharp linear phase once
the primal rank is identified (dotted line). (b)--(c) One hundred runs on the quotient
relaxation of \cref{sec:degeneracy} of a second instance ($N=211$), each
started at a small random perturbation of a sampled singular optimal
endpoint $\bar{Z}_i$ (\cref{app:experiments}): the endpoints are singular
and mutually distinct (b), and every run converges locally linearly (c).}
\label{fig:motivation}
\end{figure}

\paragraph{Two unusual observations.}
Practice, however, does not match this pessimism, and
\cref{fig:motivation} shows the phenomenon that motivates this paper. On
the second-order moment relaxation of a BQP instance with $n=20$, exact
ADMM started cold converges to a KKT point whose certificate is \emph{not}
strictly complementary, yet the fixed-point residual still enters a sharp
linear phase (\cref{fig:motivation}(a)). This behavior was recorded as an
open question in~\cite{kang25arxiv-admmsdp-linearconv}. Nor is it an
accident of one trajectory. On a second quotient instance, one hundred
runs initialized near sampled \emph{singular} optimal fixed points
(\cref{fig:motivation}(b)) all converge locally linearly
(\cref{fig:motivation}(c)). A dual Riemannian method reports similar behavior on BQP
relaxations~\cite{wang25arxiv-dual-riemannian-admm-lowrank-sdp}, and the
effect persists for higher-degree objectives (\cref{app:experiments}). For general degenerate SDPs, by contrast, the
pessimism is justified by explicit examples that lose local linear
convergence
altogether~\cite{jiang26arxiv-linearconv-pdhg-sdp,cui16arxiv-superlinear-alm-sdp,kang26arxiv-admmsdp-limitdyn,kang26arxiv-afterimage}.
Exact Boolean relaxations thus appear to carry a hidden benign structure,
which this paper identifies and proves.

\paragraph{Main message: the degeneracy is benign.}
Our headline result explains these observations: for exact Boolean
Moment--SOS relaxations in the quotient form of
\cref{fig:motivation}(b)--(c), constructed in \cref{sec:degeneracy}, one
strictly complementary certificate---anywhere in $\Sopt$---is enough for
the fast regime. The reason is a uniform \emph{setwise} error bound for
the canonical Douglas--Rachford (DR) residual around the entire
fixed-point set. It is produced by a mechanism we call \emph{residual-field
rigidity}: a positive first-order modulus of the residual field over the
relative interior of the optimal set cannot vanish at its singular
boundary. Equivalently, the directions along which the residual stalls to
first order are exactly the directions along the optimal set, so the
second-order slow regions of~\cite{kang26arxiv-admmsdp-limitdyn} cannot
form near it. Existing local linear-rate theory for splitting methods on
SDPs and related conic programs is local to a regular KKT point. It
assumes strict
complementarity, nondegeneracy, or an error bound at the limit point the
algorithm
selects~\cite{han18mor-linear,kang25arxiv-admmsdp-linearconv,jiang26arxiv-linearconv-pdhg-sdp,ding26arxiv-linearconv-fom-cp,li26arxiv-gpu-conic-qp-linearconv-sc}.
Augmented Lagrangian methods are different, since one strictly
complementary pair already yields the objective-level error bound they
need~\cite{liao24neurips-inexact-alm,cui16arxiv-superlinear-alm-sdp}; the
splitting residual does not inherit that bound (\cref{rem:tied-objectives}).
Strict complementarity and nondegeneracy fail at the singular endpoints
of a positive-dimensional certificate set. Rigidity replaces them by a condition
that is checked at interior endpoints only; the regularity is
\emph{imported} from there to the whole set.
For the Boolean question above, unique rank-one exactness plus one
strictly complementary certificate put the quotient relaxation in the fast
regime of first-order methods; what remains open is the entrance time,
which separates this rate statement from a complexity statement
(\cref{sec:discussion}). (Panel (a) uses the standard encoding; see
\cref{sec:discussion}.)

\begin{theorem}[Benign degeneracy, informal]\label{thm:informal-A}
Fix an order $d$ at which the quotient Boolean moment SDP of
\cref{sec:degeneracy} has a unique rank-one optimum and at least one
strictly complementary optimal SOS certificate.
Then, for every admissible step choice, the canonical DR residual
satisfies a uniform local error bound around its entire fixed-point set.
Consequently, ADMM, sGS--ADMM, and PDHG converge
eventually linearly from every initialization, and the limiting KKT point
they select need not be strictly complementary. Moreover, outside a
lower-dimensional semialgebraic set in the $2^n$-dimensional
space of Boolean objectives, both hypotheses hold simultaneously at every
uniquely rank-one exact order. The formal statements are
\cref{thm:boolean-fixed,thm:main} and \cref{cor:admm,cor:pdhg}.
\end{theorem}

\paragraph{Contributions.}
Our contributions are threefold.
\begin{enumerate}[label=\textup{(\Roman*)},leftmargin=2.6em]
\item \textbf{Residual-field rigidity} (\cref{sec:rigidity}). For a
semidefinite program with compact optimal sets and one
strictly complementary optimal pair, where neither optimum need be
unique, we prove that five properties of the Douglas--Rachford residual
are equivalent. They are a uniform relative-interior derivative modulus,
two finite kernel identities on the relative boundaries of the two
optimal sets, a setwise error bound around the fixed-point set, a uniform modulus
along every boundary normal, and exact first-order linearization at every endpoint. The proof uses exact projection
secants instead of Taylor expansion at moving endpoints, and
fixed anchors instead of spectral gaps. The theorem is Boolean-free. To
our knowledge, it is the first error bound for a splitting residual on the
semidefinite cone that holds uniformly around a positive-dimensional
optimal set with singular boundary points;
earlier bounds cover polyhedral, piecewise linear-quadratic, and
well-posed conic feasibility
problems~\cite{pena21arxiv-dr-error-bound,aspelmeier16siims-admm-dr-local-linear}
or concern the objective rather than the splitting
residual~\cite{liao24neurips-inexact-alm,cui16arxiv-superlinear-alm-sdp}.
\item \textbf{Boolean verification} (\cref{sec:boolean}). We verify the
hypothesis of the rigidity theorem for exact Boolean Moment--SOS
relaxations: two quotient-space identities---a constant-diagonal identity
and an XOR-class identity---reduce the required relative-interior modulus to
a finite-dimensional kernel computation and a compactness argument.
\item \textbf{Algorithmic transfer and genericity}
(\cref{sec:algorithms,sec:genericity}). We transfer the canonical error
bound to ADMM, sGS--ADMM, and PDHG through residual
rescalings and an affine endpoint chart; with averagedness in
method-specific metrics, this yields the rates of \cref{thm:informal-A}. A separate semialgebraic argument
establishes its generic half.
\end{enumerate}
To our knowledge, the geometry of the exact regime at orders $d\ge2$ has
not been used to analyze first-order solvers; \cref{app:related} organizes
the related work into four threads.

\paragraph{Organization.}
\Cref{app:related} discusses related work,
\crefrange{app:formulation}{app:genericity} carry the deferred proofs, and
\cref{app:experiments} details the experiments.

\paragraph{Notation.}
All spaces are finite-dimensional and real. $\Sym^p$ is the space of
symmetric $p\times p$ matrices with the trace inner product
$\ip{H_1}{H_2}=\tr(H_1H_2)$, Frobenius norm $\normF{\cdot}$, and PSD cone
$\Sym_+^p$; $\Pi_+$ is the Frobenius projection onto $\Sym_+^p$ and
$\Pi_-:=\Id-\Pi_+$. For $S\subseteq[n]$, $x^S:=\prod_{i\in S}x_i$; optima
carry star subscripts ($x_\star$, $X_\star$), optimal sets are calligraphic
($\Sopt$, $\Zopt$), and we distinguish the number of Boolean variables $n$,
a general SDP matrix size $p$, and the moment matrix size $N_d$ (with
$p=N_d$ in the Boolean specialization). Directional derivatives $F'(z;h)$, the cones $T_C(z)$, $N_C(z)$, the
operators $\ri$, $\rb$, $\aff$, $\cl$, and the vector and operator norms
are standard and collected in \cref{app:spectral}.

%% file: sections/degeneracy.tex
\section{Degeneracy of Exact Boolean Relaxations}
\label{sec:degeneracy}

This section makes the degeneracy phenomenon concrete; full proofs are in
\cref{app:formulation}.

\paragraph{Boolean quotient moments.}
Fix $1\le d\le n$ and the Boolean ideal
$I_{\mathrm B}:=\langle x_1^2-1,\ldots,x_n^2-1\rangle$, so that every
polynomial has a unique square-free representative and
$x^Ax^B\equiv x^{A\triangle B}$, where $\triangle$ is the symmetric
difference. Index monomials by
\begin{equation}\label{eq:index-sets}
  \cI_d:=\{T\subseteq[n]:|T|\le\min\{d,n\}\}, \quad 
  N_d:=|\cI_d|,\quad
  b_d(x):=(x^T)_{T\in\cI_d}\in\R^{N_d};
\end{equation}
when $d$ is fixed we write $N=N_d$. For $\Gamma\in\cI_{2d}$, define the
XOR-class matrices $(A_\Gamma)_{A,B}:=\mathbf 1_{\{A\triangle B=\Gamma\}}$,
indexed by $A,B\in\cI_d$. They have disjoint supports, are Frobenius
orthogonal, and $A_\emptyset=I_N$. A degree-$2d$ pseudo-moment sequence
$y=(y_\Gamma)_{\Gamma\in\cI_{2d}}$ is encoded by the moment matrix
\begin{equation}\label{eq:moment-matrix}
  M_d(y):=\sum_{\Gamma\in\cI_{2d}}y_\Gamma A_\Gamma,
  \qquad\text{i.e.}\quad (M_d(y))_{A,B}=y_{A\triangle B},
\end{equation}
and the feasible \emph{moment body} is the spectrahedron
$\cX_d:=\{M_d(y):y_\emptyset=1,\ M_d(y)\succeq0\}$. Writing
$C_f:=\sum_\Gamma f_\Gamma A_\Gamma/\normF{A_\Gamma}^2$ for the canonical
matrix representative of $f=\sum_\Gamma f_\Gamma x^\Gamma$ (so that
$\ip{C_f}{M_d(y)}=\sum_\Gamma f_\Gamma y_\Gamma$), at every admissible
order ($\deg f\le2d$) the relaxation of \eqref{eq:pbo} is
\begin{equation}\label{eq:moment-primal}
  (\mathrm P_d)\qquad p_d(f):=\min_{X\in\cX_d}\ \ip{C_f}{X}.
\end{equation}
The identity matrix is strictly feasible and $\cX_d$ is compact, so the
minimum is attained. Moreover, $(\mathrm P_d)$ is a standard-form SDP under
an explicit independent equality encoding $\cA_dX=e_1$ of full row rank
$m=N(N+1)/2-|\cI_{2d}|+1$ (\cref{prop:encoding}). Its exact dual is the
SOS problem: maximize $\gamma$ subject to
$f-\gamma\equiv b_d(x)^\top Sb_d(x)\pmod{I_{\mathrm B}}$ with the Gram
matrix $S\succeq0$ (\cref{app:formulation}). When the
relaxation is exact with value $f_\star=\OPT(f)$, the optimal Gram
certificates form the set
\begin{equation}\label{eq:slack-set}
  \Sopt:=\{S\succeq0:\ b_d(x)^\top Sb_d(x)\equiv f-f_\star
  \ (\mathrm{mod}\ I_{\mathrm B})\},
\end{equation}
which is compact because every $S\in\Sopt$ has the same trace
$f_\emptyset-f_\star$ (\cref{lem:slack-trace}).

\paragraph{Unique rank-one exactness.}
If $f$ has a unique Boolean minimizer $x_\star$, define the first order of
\emph{unique rank-one exactness},
\begin{equation}\label{eq:dstar}
  d_\star(f):=\min\Bigl\{1\le d\le n:\ \deg f\le2d,\quad
  \argmin(\mathrm P_d)=\{b_d(x_\star)b_d(x_\star)^\top\}\Bigr\}.
\end{equation}
This is well defined with $d_\star(f)\le n$, and it persists upward: unique
rank-one exactness at $d_0$ implies it at every $d_0\le d\le n$
(\cref{prop:eventual}). Throughout the Boolean sections,
\cref{sec:degeneracy} and \crefrange{sec:boolean}{sec:algorithms}, we fix
an order $d\ge d_\star(f)$ and abbreviate $X_\star=uu^\top=Nqq^\top$ with
$u:=b_d(x_\star)\in\{\pm1\}^N$ and $q:=u/\sqrt N$.

\begin{remark}[Symmetric objectives need a gauge]\label{rem:gauge}
A unique Boolean minimizer excludes symmetric models such as MAX-CUT or
zero-field Ising objectives, whose minimizers come in $\pm$ pairs; they
enter after a standard gauge fixing (fixing one spin), and we use them as
motivation only.
\end{remark}

\paragraph{Merged squares and degeneracy.}
Every optimal certificate annihilates $u$, since
$u^\top Su=f(x_\star)-f_\star=0$ and $S\succeq0$; so $\Sopt$ lives on the
cone $\Sym_+(V):=\{S\succeq0:\ran(S)\subseteq V\}$, where
$U:=\ran(X_\star)=\spanof\{q\}$ and $V:=\ker(X_\star)=q^\perp$. Now let
$d\ge2$, choose two coordinates $i\ne j$, put
$a:=x_{\star,i}x_{\star,j}$, and define two vectors in $\R^N$:
$w_1:=e_\emptyset-a\,e_{\{i,j\}}$ and $w_2:=e_{\{i\}}-a\,e_{\{j\}}$.
Both are orthogonal to $u$, and because
$(x^{\{i,j\}})^2\equiv(x^{\{i\}})^2\equiv1$, both rank-one Gram matrices
represent the \emph{same} Boolean polynomial:
\begin{equation}\label{eq:same-poly}
  b_d(x)^\top w_1w_1^\top b_d(x)
  \equiv 2\bigl(1-a\,x_ix_j\bigr)
  \equiv b_d(x)^\top w_2w_2^\top b_d(x)
  \pmod{I_{\mathrm B}}.
\end{equation}
Quotienting by $x_i^2=1$ thus merges distinct squares and creates nonzero
\emph{same-polynomial directions} such as $w_1w_1^\top-w_2w_2^\top$. These
directions live in the subspace $\cL:=\ran(\cA_d^*)\cap\Sym(V)$, where
$\Sym(V)$ denotes the symmetric matrices supported on $V$. Differences of
optimal certificates always lie in $\cL$, and under strict complementarity
$\cL$ is the direction space of $\Sopt$ (\cref{lem:product-geometry}).

\begin{proposition}[Primal degeneracy at every order $d\ge2$]\label{prop:degeneracy}
For the independent encoding $\cA_d$, primal nondegeneracy of
$X_\star$~\cite{alizadeh9mp-complementarity-nondegeneracy} holds if and only
if $m=N$, and $\dim \cL=m-N=\frac{N(N-1)}2-|\cI_{2d}|+1$. At $d=1$ one has
$m=N$; at every $2\le d\le n$ one has $m>N$
(indeed $w_1w_1^\top-w_2w_2^\top\in \cL\setminus\{0\}$), so the rank-one
primal optimum is degenerate. If, in addition, $\Sopt$ contains a
\emph{strictly complementary} certificate $\Ssc$, i.e.\
$\rank(\Ssc)=N-1$, then $\aff\Sopt=\Ssc+\cL$, so $\Sopt$ is
positive-dimensional with $\dim\Sopt=m-N$.
\end{proposition}

The proof is in \cref{app:formulation}. Degeneracy here concerns the
optimum relative to the constraints rather than singularity of the matrix
$X_\star$ itself. The encoding is always \emph{injective} on the tangent
space of the rank-one PSD manifold, while nondegeneracy asks for
\emph{surjectivity} onto $\R^m$.

\paragraph{The geometric problem.}
Suppose from now on that $\Sopt$ contains one strictly complementary
certificate $\Ssc$; equivalently, $X_\star+\Ssc\succ0$. Then
$\Sopt=(\Ssc+\cL)\cap\Sym_+(V)$ (\cref{lem:product-geometry}), a compact convex
slice of the PSD cone supported on $V$. The certificate $\Ssc$ lies in the
relative interior of this slice, every certificate on its relative boundary
loses rank, and by \cref{prop:degeneracy} the boundary is nonempty for
$d\ge2$. A first-order method is free to converge to any point of the
slice; at a relative-boundary point, analyses that require strict
complementarity at the selected limit do not apply, and general conic
programming \emph{does} produce arbitrarily slow regions
there~\cite{kang26arxiv-afterimage}.

%% file: sections/rigidity.tex
\section{Residual-Field Rigidity}
\label{sec:rigidity}

This section is Boolean-free. \Cref{ass:two-sided} below asks only for
compact optimal sets and one strictly complementary pair; neither optimum
needs to be unique. The result is a theorem that makes five properties of
the residual field equivalent, among them a setwise error bound.
\Cref{sec:boolean} verifies one of them and inherits the other four.

\subsection{The canonical residual and the two-sided geometry}
\label{sec:rigidity:setup}

Consider the primal--dual pair
\begin{equation}\label{eq:sdp-pair}
 \begin{aligned}
 (\mathrm P)\quad &\min_{X\succeq0}\ \ip CX
                 &&\text{s.t. }\cA X=b,\\
 (\mathrm D)\quad &\max_{\lambda\in\R^m,\,S\succeq0}\ b^\top\lambda
                 &&\text{s.t. }\cA^*\lambda+S=C,
 \end{aligned}
\end{equation}
with $\cA:\Sym^p\to\R^m$ surjective; in the dual, only the slack $S$ is
constrained to be PSD. Set
$\cR:=\ran(\cA^*)$, $\cN:=\ker(\cA)=\cR^\perp$, and
$\cA^\dagger:=\cA^*(\cA\cA^*)^{-1}$, and let $\Prow:=\cA^\dagger\cA$ and
$\Pnull:=\Id-\Prow$ be the orthogonal projectors onto $\cR$ and $\cN$ on
the \emph{matrix} space $\Sym^p$; put $X_b:=\cA^\dagger b$. For fixed
$\sigma>0$, define the two orthogonal residual
components and the canonical Douglas--Rachford (DR) map
\begin{equation}\label{eq:res-components}
 \rpar(Z):=\Prow\bigl(\Pi_+(Z)-X_b\bigr),\quad
 \rper_\sigma(Z):=\Pnull\bigl(\Pi_-(Z)+\sigma C\bigr),\quad
 \delta_\sigma:=-\rpar-\rper_\sigma,\quad
 T_{\sigma,1}:=\Id+\delta_\sigma .
\end{equation}
This is the DR map (with $\Pi_+$ applied before the affine step) for
$F(X)=\ip CX+\iota_{\{\cA X=b\}}(X)$ and $G=\iota_{\Sym_+^p}$; its fixed
points are $X-\sigma S$ over all KKT pairs, with $X=\Pi_+(Z)$ and
$S=-\sigma^{-1}\Pi_-(Z)$ (\cref{lem:drs}, \cref{app:drs-slice}).

\begin{assumption}[Two-sided optimal geometry]\label{ass:two-sided}
The primal optimal set $\Xopt$ and the optimal slack set $\Sopt$ are
nonempty and compact, and there is an optimal pair $(\Xsc,\Ssc)$ with
$\Xsc\Ssc=0$, $\Xsc+\Ssc\succ0$, and $1\le\rank\Xsc<p$.
\end{assumption}

We call $(\Xsc,\Ssc)$ the \emph{anchor pair}, fixed once and for all;
the uniform spectral estimates below rest on it, never on the moving
optima. Set $U:=\ran\Xsc$ and
$V:=\ran\Ssc=U^\perp$, let $P_U$, $P_V$ be the orthogonal projectors onto
them, write $\Sym(U)$, $\Sym(V)$ for the symmetric matrices supported on
$U$, $V$, and define the primal and dual direction spaces and their sum
\begin{equation}\label{eq:directions}
 \cM:=\cN\cap\Sym(U),\qquad \cL:=\cR\cap\Sym(V),\qquad \cD:=\cM\oplus \cL .
\end{equation}
Every optimal $X$ is supported on $U$ and every optimal $S$ on $V$, so
(\cref{lem:product-geometry}, \cref{app:drs-slice})
\begin{equation}\label{eq:product}
 \Xopt=(\Xsc+\cM)\cap\Sym_+(U),\quad
 \Sopt=(\Ssc+\cL)\cap\Sym_+(V),\quad
 \Zopt=\Fix T_{\sigma,1}=\Xopt-\sigma\Sopt .
\end{equation}
We call the points of $\Zopt$ \emph{endpoints}, where convergent runs
end. Each endpoint $\bar Z$ has a unique representation
$\bar X-\sigma\bar S$, and we call $\bar X$ and $\bar S$ its two
\emph{factors}. Also $\aff\Zopt=\Xsc-\sigma\Ssc+\cD$, and
$\bar Z\in\ri\Zopt$ exactly when $\bar X$ is positive definite on $U$ and
$\bar S$ on $V$. In words, $\Zopt$ is a product of two PSD slices, and
its relative boundary consists of the endpoints at which a factor loses
rank. When $\Xopt=\{X_\star\}$ is a singleton, $\cM=\{0\}$ and $\cD=\cL$,
which is the Boolean situation of \cref{sec:degeneracy}.

\subsection{Endpoints, derivatives, and the stall cone}
\label{sec:rigidity:spectral}

Fix an endpoint $\bar Z=\bar X-\sigma\bar S$ and set
$K_{\bar X}:=\ker\bar X\cap U$ and $K_{\bar S}:=\ker\bar S\cap V$, the
subspaces on which the two factors have lost rank; both are $\{0\}$
exactly on $\ri\Zopt$. Only these kernels move with the endpoint; $U$ and
$V$ stay fixed, and no continuity of eigenbases in the endpoint is ever
assumed. At an interior endpoint the PSD projection is
differentiable, and in the fixed $U\oplus V$ blocks its derivative
$J_{\bar X,\bar S}(H):=\Pi_+'(\bar Z;H)$ keeps the $UU$ block, kills the
$VV$ block, and replaces the cross block by the solution $Y$ of the
Sylvester equation
\begin{equation}\label{eq:sylvester-main}
 \bar X_{UU}\,Y+\sigma\,Y\,\bar S_{VV}=\bar X_{UU}H_{UV}
\end{equation}
(\cref{app:deriv-validity}). Writing $J$ for $J_{\bar X,\bar S}$, the
residual derivative is $\delta_\sigma'(\bar Z;H)=-\Prow J(H)-\Pnull(H-J(H))$,
with orthogonal components. Two cones organize everything that follows:
\begin{equation}\label{eq:stall-cone}
 \cC(\bar Z):=\{H\in\Sym^p:\delta_\sigma'(\bar Z;H)=0\},\qquad
 \Tan_{\Zopt}(\bar Z):=\text{the tangent cone of }\Zopt\text{ at }\bar Z .
\end{equation}
The first is the \emph{stall cone}: the directions along which the
residual vanishes to first order, so the first-order dynamics of the
splitting method stalls~\cite{kang26arxiv-admmsdp-limitdyn}; the second
collects the directions along the optimal set. Always
$\Tan_{\Zopt}(\bar Z)\subseteq\cC(\bar Z)$, since $\delta_\sigma$ vanishes
on $\Zopt$; the question is whether the stall cone is \emph{larger}
(\cref{app:limit-dynamics} relates this to the second-order limit dynamics
of~\cite{kang26arxiv-admmsdp-limitdyn}).

\subsection{The rigidity theorem}
\label{sec:rigidity:thm}

Define the relative-interior and full-normal derivative moduli
\begin{equation}\label{eq:mu-def}
 \mu_{\mathrm{ri}}:=
 \inf_{\substack{\bar Z\in\ri\Zopt\\ H\in\cD^\perp,\ \normF H=1}}
        \normF{\delta_\sigma'(\bar Z;H)},
 \qquad
 \mu_{\mathrm{nor}}:=
 \inf_{\substack{\bar Z\in\Zopt\\ H\in N_{\Zopt}(\bar Z),\ \normF H=1}}
        \normF{\delta_\sigma'(\bar Z;H)} .
\end{equation}
\begin{theorem}[Two-sided rigidity and exact linearization]\label{thm:rigidity}
Under \cref{ass:two-sided}, the following are equivalent:
\begin{enumerate}[label=\textup{(\roman*)},leftmargin=2.2em]
\item \emph{Uniform interior modulus:} $\mu_{\mathrm{ri}}>0$.
\item \emph{Complementarity rigidity:} the \emph{primal identity} (left)
holds for every $\bar S\in\rb\Sopt$ and the \emph{dual identity} (right)
for every $\bar X\in\rb\Xopt$, both named after the direction space they
test:
\begin{equation}\label{eq:kernel-tests}
 \{E\in\cN:\ P_VEP_V=0,\ E\bar S=0\}=\cM,\qquad
 \{F\in\cR:\ P_UFP_U=0,\ \bar XF=0\}=\cL .
\end{equation}
\item \emph{Setwise error bound:} there are $c_\sigma,t_\sigma>0$ with
\begin{equation}\label{eq:canonical-eb}
 \normF{\delta_\sigma(Z)}\ \ge\ c_\sigma\,\dist_F(Z,\Zopt)
 \qquad\text{whenever }\dist_F(Z,\Zopt)\le t_\sigma .
\end{equation}
\item \emph{Uniform normal modulus:} $\mu_{\mathrm{nor}}>0$.
\item \emph{Exact linearization:} $\cC(\bar Z)=\Tan_{\Zopt}(\bar Z)$ for
every $\bar Z\in\rb\Zopt$.
\end{enumerate}
Conditions (ii) and (v) hold automatically at relative-interior factors
and endpoints, so requiring them everywhere is equivalent. Whether the
conditions hold does not depend on $\sigma$; the constants do. No
positive lower bound on the nonzero eigenvalues of the varying optimal
pair $(\bar X,\bar S)$ is imposed.
\end{theorem}

\paragraph{What each condition says.}
\emph{(i) Stiffness across the flat interior.} Transverse to the optimal
set, the residual responds at first order, with a lower bound
$\mu_{\mathrm{ri}}$ that is uniform over the interior endpoints, however
close they sit to the singular boundary. This is the
verifiable condition. Interior endpoints are nonsingular matrices, where
the residual is real-analytic with the explicit linear derivative
\eqref{eq:sylvester-main}, so checking (i) is linear algebra plus
compactness; \cref{sec:boolean} does this.
\emph{(ii) Complementarity rigidity.} At an endpoint where
one factor has lost rank, a direction can couple its vanished eigenspace
with the support of the other factor, and the linearized residual
does not see it. The two identities say that every such direction
already runs along the optimal set.
\emph{(iii) The residual measures the distance.} In a tube around the
optimal set, the KKT residual is at least a constant times the distance
to it, which averaged splitting methods convert into linear rates
(\cref{sec:algorithms}).
\emph{(iv) Stiffness along every normal.} At a singular endpoint the
normal cone is larger than $\cD^\perp$, since directions leaving the PSD
cone through a vanished eigenspace are normal there too; (iv) asks for
the response of (i) along all these normals, with a bound uniform
over all endpoints.
\emph{(v) Stall cone equals tangent cone.} The directions along which
the residual vanishes to first order are exactly those along the optimal
set; equivalently, no direction leaves $\Zopt$ while the first-order
dynamics stalls---the geometric content of ``benign''.

\subsection{Proof outline of~\cref{thm:rigidity}}
\label{sec:rigidity:proof}

\paragraph{Two obstacles.}
A pointwise argument would differentiate the residual at the nearest
endpoint; two features of the degenerate geometry defeat this. First, the \emph{base moves}. Along a
sequence $Z_k\to\Zopt$ the nearest endpoint $\bar Z_k$ changes with $k$,
and the projection derivative is discontinuous in the endpoint exactly
where a kernel grows. So a Taylor expansion at $\bar Z_k$ has no uniform
remainder. Second, \emph{eigenvalues vanish}. The positive eigenvalues of
the moving optima $\bar X_k$, $\bar S_k$ can tend to zero, so any
bound whose constant contains their reciprocals degenerates along the
sequence.

\paragraph{Proof outline.}
We prove (i)$\Rightarrow$(ii)$\Rightarrow$(iii)$\Rightarrow$(iv)$\Rightarrow$(i)
and (iii)$\Rightarrow$(v)$\Rightarrow$(ii). The step
(ii)$\Rightarrow$(iii) carries the weight, and three devices handle the
two obstacles.

\emph{Fixed anchors.} Every uniform spectral estimate comes from the anchor pair,
through $a_X:=\lambda_{\min}(\Xsc|_U)$ and $a_S:=\lambda_{\min}(\Ssc|_V)$.
The Sylvester operator
$Y\mapsto\Xsc|_U\,Y+\sigma Y\bar S_{VV}$ has inverse norm at most
$a_X^{-1}$ for every optimal $\bar S$, singular or not, and symmetrically
with $(\sigma a_S)^{-1}$ (\cref{lem:sylvester}). Together with the
optimal-set radii of \eqref{eq:anchor-constants}, these gaps bound every
decomposition of a normal vector into a $\cD^\perp$ part and two kernel
parts (\cref{lem:normal-bound}). The anchors also prove
(i)$\Rightarrow$(ii): to test the primal identity at a boundary
certificate $\bar S$, form the interior endpoints $\Xsc-\sigma\bar S_\eps$
with $\bar S_\eps=(1-\eps)\bar S+\eps\Ssc$, whose primal factor is the
anchor and whose dual factor is pushed toward its anchor. Along the push,
the interior derivative of a direction $E$ from the left side of
\eqref{eq:kernel-tests} converges to $E$ itself. Hence the residual
derivative tends to zero, and (i) forces $E\in\cD$; the dual
identity is symmetric (\cref{app:rigidity-tests}).

\emph{Exact secants.} For (ii)$\Rightarrow$(iii) we take $Z_k$ with
$t_k:=\dist_F(Z_k,\Zopt)\to0$ and residual $o(t_k)$, and we never
differentiate. With $\bar Z_k$ the nearest endpoint, the exact projections
$X_k=\Pi_+(Z_k)$ and $S_k=-\sigma^{-1}\Pi_-(Z_k)$ define the secants
$E_k=(X_k-\bar X_k)/t_k$ and $F_k=-\sigma(S_k-\bar S_k)/t_k$, which are
bounded and sum to the unit normal $H_k=(Z_k-\bar Z_k)/t_k$. These exact
secants avoid a Taylor remainder at the moving endpoint. The small residual
says $\Prow E_k\to0$ and $\Pnull F_k\to0$, so $\ip{E_k}{F_k}\to0$. Exact
complementarity $X_kS_k=0$ gives the energy identity
$\ip{E_k}{F_k}=\sigma t_k^{-2}(\ip{X_k}{\bar S_k}+\ip{\bar X_k}{S_k})\ge0$.
Hence both complementarity gaps are $o(t_k^2)$.

\emph{The PSD product inequality.} For PSD $A,B$,
$\normF{AB}^2\le\normop A\normop B\,\ip AB$ (\cref{lem:psd-product}).
Applied to the two gaps, it gives $\normF{X_k\bar S_k}=o(t_k)$ and
$\normF{\bar X_kS_k}=o(t_k)$, that is, $E_k\bar S_k\to0$ and
$\bar X_kF_k\to0$. No eigenvalue bound enters here. 
The anchors remove the blocks on the wrong support, $(E_k)_{VV}\succeq0$
and $(F_k)_{UU}\preceq0$, since these pair to $o(1)$ with $\Ssc$ and
$\Xsc$. The limits $E,F$ are therefore exactly the directions tested in
\eqref{eq:kernel-tests}, so $E\in\cM$, $F\in\cL$, and the limit $H=E+F$ of
the unit normals lies in $\cD$. The anchor-controlled normal decomposition
and the signs of the secants against the kernel cones then give a
contradiction (\cref{app:rigidity-eb}). The remaining implications are short
(\cref{app:rigidity-linearization}). Differentiating the error bound along
a normal gives (iii)$\Rightarrow$(iv), and the inclusion
$\cD^\perp\subseteq N_{\Zopt}$ at interior endpoints gives
(iv)$\Rightarrow$(i). For (iii)$\Rightarrow$(v), a residual of order $o(t)$
along $\bar Z+tH$ forces $\dist_F(\bar Z+tH,\Zopt)=o(t)$ through
\eqref{eq:canonical-eb}, so $H$ is tangent. For (v)$\Rightarrow$(ii), the
projection derivative is computed at the two anchored endpoints, where it
fixes primal test directions and annihilates dual ones. The PSD structure
is essential (\cref{ex:not-abstract,ex:two-disks}).

%% file: sections/boolean.tex
\section{Boolean Verification of the Modulus}
\label{sec:boolean}

We now verify condition (i) of \cref{thm:rigidity}, the uniform interior
modulus, for the Boolean moment SDP of \cref{sec:degeneracy}; omitted
computations are in \cref{app:modulus}. Fix an order $d$ at which the
relaxation is uniquely rank-one exact with a strictly complementary
certificate, and instantiate \cref{sec:rigidity} with $p=N$,
$\Xsc=X_\star=Nqq^\top$, $U=\spanof\{q\}$, $V=q^\perp$, $\cA=\cA_d$,
$b=e_1$, $C=C_f$, and $X_b=I_N$. Here $\cA_d$ is the independent encoding
of \cref{prop:encoding}: its null space is the XOR-consistency direction
space $\cT_d:=\spanof\{A_\Gamma:\emptyset\ne\Gamma\in\cI_{2d}\}$ (every
matrix in it has zero diagonal), and its row space is
$\cR_d:=\cT_d^\perp$. Compactness of $\Sopt$ is the trace identity
(\cref{lem:slack-trace}), and $\Xopt=\{X_\star\}$ is a point, so
\cref{ass:two-sided} holds with $\cM=\{0\}$ and $\cD=\cL$, and the dual
identity in \eqref{eq:kernel-tests} is automatic.

\subsection{Two decisive subspace identities}

\begin{lemma}[Boolean subspace geometry]\label{lem:boolean-geometry}
With $V=q^\perp$:
\begin{align}
 \Sym(V)^\perp&=\{qv^\top+vq^\top:v\in\R^N\},\label{eq:Fq-perp}\\
 \cT_d\cap\Sym(V)^\perp&=\{0\},\label{eq:transverse-intersection}\\
 \cL=\Sym(V)\cap\cR_d&=\Sym(V)\cap\cH_d^\perp,
 \qquad \cH_d:=\spanof\{I_N\}\oplus\cT_d.\label{eq:boolean-L}
\end{align}
\end{lemma}

The first identity is elementary. The second is the \emph{constant-diagonal
identity}: a matrix $qv^\top+vq^\top\in\cT_d$ has zero diagonal $2q_Av_A$
with every $q_A=\pm N^{-1/2}$ nonzero, so $v=0$. The third is the
\emph{XOR-class identity}, built on
$uu^\top=\sum_{\Gamma\in\cI_{2d}}x_\star^\Gamma A_\Gamma$; the full proofs
are in \cref{app:identities-proof}. In words: no consistency direction can
lie in $\Sym(V)^\perp$, and inside $\Sym(V)$ the equality row space and
the Gram-map kernel agree.

\subsection{A continuous extension of the interior derivative}

For every $S\in\Sopt$---including boundary certificates---define the linear
maps
\begin{equation}\label{eq:Jextension}
\begin{gathered}
 \Jext_S(H):=h_{00}\,qq^\top+qw^\top+wq^\top,
 \qquad h_{00}:=q^\top Hq,\qquad
 w:=N(NI_N+\sigma S)^{-1}P_VHq,\\
 \Psi_S:=\Prow\Jext_S+\Pnull(\Id-\Jext_S) .
\end{gathered}
\end{equation}
The inverse exists because $NI_N+\sigma S\succeq NI_N$, and
$S\mapsto\Psi_S$ is operator-norm continuous on all of $\Sopt$. For
$S\in\ri\Sopt$ the projection is differentiable at $X_\star-\sigma S$ and
$\Pi_+'(X_\star-\sigma S;\cdot)=\Jext_S$, so
$\delta_\sigma'(X_\star-\sigma S;\cdot)=-\Psi_S$. At a boundary certificate
the true directional derivative acquires exactly one extra nonlinear term,
$Q_K\Pi_+(Q_K^\top HQ_K)Q_K^\top$ with $Q_K$ an orthonormal basis of
$K:=\ker(S)\cap V$ (\cref{app:deriv-id}). So $\Psi_S$ is a continuous
\emph{algebraic extension} of the interior derivative.

\begin{proposition}[Uniform relative-interior modulus]\label{prop:boolean-modulus}
For every $S\in\Sopt$, $\ker(\Psi_S)=\cL$; consequently
\begin{equation}\label{eq:boolean-muri}
 \mu_{\mathrm{ri}}
 \ \ge\ \mu_\star:=\min_{S\in\Sopt}\ \sigma_{\min}\bigl(\Psi_S|_{\cL^\perp}\bigr)
 \ >\ 0 .
\end{equation}
\end{proposition}

The proof (\cref{app:modulus-proof}) is a kernel chain through the two
identities. If $\Psi_S(H)=0$, the two residual components vanish
separately. The constant-diagonal identity
\eqref{eq:transverse-intersection} then forces
$\Jext_S(H)\in\cT_d\cap\Sym(V)^\perp=\{0\}$, and unwinding
\eqref{eq:Jextension} places $H$ in $\Sym(V)\cap\cR_d=\cL$. Continuity
of $S\mapsto\Psi_S$ and compactness of $\Sopt$ turn the pointwise kernel
statement into the uniform bound \eqref{eq:boolean-muri}.

\subsection{The deterministic theorem}

\begin{theorem}[Deterministic Boolean benign degeneracy]\label{thm:boolean-fixed}
Suppose the order-$d$ Boolean moment relaxation has a unique rank-one
optimum and one strictly complementary optimal SOS certificate. Then, for
every fixed $\sigma>0$, the canonical DR residual satisfies the setwise
error bound \eqref{eq:canonical-eb} around its entire fixed-point set
$\Zopt$, and its stall cone equals the tangent cone of $\Zopt$ at every
endpoint. The algorithmic conclusions of \cref{cor:admm,cor:pdhg} hold for
the independent encoding \eqref{eq:Ad}, at every endpoint,
regardless of the rank of the certificate selected by the iteration.
\end{theorem}
\begin{proof}
\Cref{prop:encoding} gives surjectivity of $\cA_d$ and
\cref{lem:slack-trace} compactness of $\Sopt$; the instantiation
above places us under \cref{ass:two-sided};
\cref{prop:boolean-modulus} verifies condition (i) of \cref{thm:rigidity},
whose condition (iii) is the error bound and condition (v) the exact
linearization. Then apply \cref{sec:algorithms}.
\end{proof}

\begin{remark}[Uniqueness cannot be dropped in the Boolean theorem]
\label{rem:tied-objectives}
\Cref{thm:rigidity} needs no uniqueness, but \cref{thm:boolean-fixed}
does. For $f(x)=1-x_1$ on $\{\pm1\}^2$ at full order, the two minimizers
$(1,\pm1)$ make both $\Xopt$ and $\Sopt$ segments, and a strictly
complementary pair exists. Take the endpoint whose primal factor is the
moment matrix of the point mass at the first minimizer. Along the cross
direction that couples the second minimizer with a nonminimizer, the
residual is of order $2t^2/\sigma$ while the distance to $\Zopt$ is
$\sqrt2\,t$ (\cref{app:tied-objectives}). The dual identity fails there,
and the direction lies in $\cC(\bar Z)\setminus\Tan_{\Zopt}(\bar Z)$. The
same happens for every nonconstant full-order objective with two or more
minimizers. Unique rank-one exactness rules out such directions.
\end{remark}

%% file: sections/algorithms.tex
\section{Consequences for Splitting Methods}
\label{sec:algorithms}

This section converts the error bound \eqref{eq:canonical-eb} into
convergence rates for exact ADMM, exact sGS--ADMM, and fixed-step PDHG;
details are in \cref{app:algorithms}.

\paragraph{ADMM and exact sGS--ADMM share one shadow map.}
Exact ADMM on the dual SDP with penalty $\sigma>0$ and multiplier step
$\tau>0$, and its exact symmetric Gauss--Seidel
variant~\cite{gabay76cma-dual-algorithm-nonlinear-variation,chen17mp-sgsadmm},
both reduce to one common recursion in the shadow coordinate
$Z^k:=\Pnull X^k+X_b-\sigma\Prow S^k-\sigma\Pnull C$. The recursion is
$Z^{k+1}=T_{\sigma,\tau}(Z^k)$ with
$T_{\sigma,\tau}-\Id=-\mathsf D_\tau(\rpar+\rper_\sigma)$ and
$\mathsf D_\tau:=\Prow+\tau\Pnull$ (established
in~\cite{kang26arxiv-afterimage}; re-derived in \cref{app:shadow-map}).
The multiplier step only rescales the two orthogonal
residual components, so, with $\delta_{\sigma,\tau}:=T_{\sigma,\tau}-\Id$,
\begin{equation}\label{eq:admm-eb}
 \normF{\delta_{\sigma,\tau}(Z)}
 \ \ge\ \min\{1,\tau\}\,\normF{\delta_\sigma(Z)}
 \ \ge\ \min\{1,\tau\}\,c_\sigma\,\dist_F(Z,\Zopt)
\end{equation}
throughout the tube of \eqref{eq:canonical-eb}: the setwise error bound
transfers verbatim.
For $0<\tau<2$, $T_{\sigma,\tau}$ is $\max\{1,\tau\}/2$-averaged in the
parameter-dependent metric $\mathsf G_\tau:=\Prow+\tau^{-1}\Pnull$, which
is equivalent to the Frobenius metric (\cref{prop:admm-metric},
\cref{app:shadow-rates}; cf.~\cite{kang26arxiv-afterimage}).

\paragraph{PDHG admits an endpoint chart.}
Fixed-step PDHG~\cite{chambolle11jmiv-pdhg} on \eqref{eq:sdp-pair}, with
state $w=(X,\lambda)$ and steps $\eta_x,\eta_y>0$, iterates
$X^+=\Pi_+(X-\eta_x(C-\cA^*\lambda))$,
$\lambda^+=\lambda+\eta_y(b-\cA(2X^+-X))$. Consider the affine change of
variables $\Phi(X,\lambda)=(Z,r)$ with $Z:=X-\eta_x(C-\cA^*\lambda)$ and
$r:=\cA X-b$. It is a bijection sending the fixed-point set
$\Wopt:=\Fix T_{\mathrm P}$ of the PDHG map $T_{\mathrm P}$ onto
$\Zopt\times\{0\}$, with $\Zopt$ the
canonical fixed-point set at parameter $\sigma=\eta_x$. In this chart the
PDHG residual \emph{factors through the
canonical one}: there is an explicit linear isomorphism
$\mathsf L_{\eta_x,\eta_y}:\cR\times\cN\times\R^m\to\Sym^p\times\R^m$ with
\begin{equation}\label{eq:pdhg-factor}
 \Phi T_{\mathrm P}\Phi^{-1}(Z,r)-(Z,r)
 =\mathsf L_{\eta_x,\eta_y}\bigl(\rpar(Z),\,\rper_{\eta_x}(Z),\,r\bigr)
\end{equation}
(\cref{prop:pdhg-transfer}). Invertibility of $\mathsf L_{\eta_x,\eta_y}$
and of the chart converts \eqref{eq:canonical-eb} into a native setwise
error bound for PDHG, and $T_{\mathrm P}$ is firmly nonexpansive in the
standard preconditioned metric whenever
$\eta_x\eta_y\normop{\cA}^2<1$~\cite{chambolle11jmiv-pdhg,oconnor20mp-pdhg-drs}
(\cref{prop:pdhg-metric}).

\paragraph{From error bound to rates.}
The passage from a setwise error bound to convergence rates is standard
averaged-operator
theory~\cite{bauschke17springer-convex-analysis-hilbert-spaces}. For an
$\alpha$-averaged map in a metric equivalent to the Euclidean one, a
setwise error bound on a neighborhood of the fixed-point set makes every
orbit converge to a fixed point. The distance then contracts eventually
Q-linearly in that metric, and the iterates converge R-linearly
(\cref{lem:averaged-eb}, stated and proved quantitatively in
\cref{app:rate-lemma}).

\begin{corollary}[ADMM and sGS--ADMM]\label{cor:admm}
Under the hypotheses of \cref{thm:boolean-fixed} (or, abstractly,
\cref{ass:two-sided} plus any condition of \cref{thm:rigidity}), for every fixed $\sigma>0$ and $0<\tau<2$, the shadow
orbits of exact ADMM and exact sGS--ADMM converge to an endpoint of
$\Zopt$. The distance contracts eventually Q-linearly in the
$\mathsf G_\tau$ metric, and the shadow and the raw primal--dual iterates
converge R-linearly (\cref{app:shadow-rates}).
\end{corollary}

\begin{corollary}[PDHG]\label{cor:pdhg}
Under the same hypotheses with $\sigma=\eta_x$ and
$\eta_x\eta_y\normop{\cA}^2<1$, every PDHG orbit converges to a KKT pair;
the distance to $\Wopt$ contracts eventually Q-linearly in the
preconditioned metric, and the iterates and the slack
$S^k=C-\cA^*\lambda^k$ converge R-linearly.
\end{corollary}

The convergence is \emph{eventual}: it begins after an entrance time that
we do not bound (\cref{app:rate-lemma}).

%% file: sections/genericity.tex
\section{Generic Benign Degeneracy}
\label{sec:genericity}

It remains to show that the two hypotheses of \cref{thm:boolean-fixed}
occur broadly: outside one lower-dimensional semialgebraic set of
objectives, they hold \emph{simultaneously at every uniquely rank-one
exact order} (proofs in \cref{app:genericity}).

\paragraph{Full order: exactness and strict complementarity for free.}
Let $\cP_n:=\spanof\{x^T:T\subseteq[n]\}$ be the $2^n$-dimensional space of
square-free objectives, and $\cP_{\le2d}$ its subspace of degree at most
$2d$.

\begin{proposition}[Full-order structure]\label{prop:eventual}
At $d=n$ the moment body $\cX_n$ is the affine image, with unique
coefficients, of the probability simplex on $\{\pm1\}^n$:
$X=\sum_{a\in\{\pm1\}^n}\pi_a\,b_n(a)b_n(a)^\top$ with $\pi_a\ge0$,
$\sum_a\pi_a=1$. Consequently every uniquely minimized $f$ is uniquely
rank-one exact at $d=n$, so $d_\star(f)\le n$; and unique rank-one exactness
persists upward from any order at which it holds. Moreover, for every
uniquely minimized $f$ the full-order relaxation admits a strictly
complementary optimal certificate of rank $2^n-1$ (\cref{app:full-order}).
\end{proposition}

\paragraph{Fixed-order cells and one Gram construction.}
Fix a candidate optimizer $a\in\{\pm1\}^n$ and an order $d$, and let
$u=b_d(a)$. Inside the normalized coefficient space
$\cQ(a,d):=\{g\in\cP_{\le2d}:g(a)=0\}$, the \emph{exact cell} $\cE(a,d)$
collects the $g$ for which $(\mathrm P_d)$ has the unique optimizer
$uu^\top$. The \emph{strictly complementary cell}
$\cE_{\mathrm{sc}}(a,d)\subseteq\cE(a,d)$ collects the $g$ represented by
some PSD Gram matrix $S$ with $\ker(S)=\spanof\{u\}$. For each nonempty
$\Gamma\in\cI_{2d}$, split
$\Gamma=P_\Gamma\dot\cup Q_\Gamma$ into parts of size at most $d$ (with
$P_\Gamma=\Gamma$, $Q_\Gamma=\emptyset$ when $|\Gamma|\le d$) and set
\begin{equation}\label{eq:wGamma}
 w_\Gamma:=e_{P_\Gamma}-a^\Gamma e_{Q_\Gamma},
 \qquad
 S_{\mathrm{int}}:=\sum_{\emptyset\ne A\in\cI_d}w_Aw_A^\top .
\end{equation}
Each $w_\Gamma$ is orthogonal to $u$ and its Gram matrix represents
$2(1-a^\Gamma x^\Gamma)$; under this splitting convention, the $N-1$ vectors
$w_A=e_A-a^Ae_\emptyset$ ($\emptyset\ne A\in\cI_d$) form a basis of
$u^\perp$, so $S_{\mathrm{int}}\succeq0$ is an \emph{interior certificate}
with kernel exactly $\spanof\{u\}$.

\begin{theorem}[Fixed-order SC genericity]\label{thm:cell-generic}
$\cE_{\mathrm{sc}}(a,d)$ is open in $\cQ(a,d)$ and dense relative to
$\cE(a,d)$; both cells are semialgebraic; and
$\dim\bigl(\cE(a,d)\setminus\cE_{\mathrm{sc}}(a,d)\bigr)<\dim\cE(a,d)$.
\end{theorem}

The proof (\cref{app:cells-proof}) gets openness from the Gram
construction \eqref{eq:wGamma}, density from $S+\eps S_{\mathrm{int}}$, and the rest
from Tarski--Seidenberg and the frontier theorem~\cite{bochnak98book-rag}.

\paragraph{Ambient genericity.}
To pass to $\cP_n$, remove two bad families. The tie
hyperplanes $\{f:f(a)=f(b)\}$, $a\ne b$, contain every objective with a
nonunique minimizer. For each $(a,d)$, the objectives with
$f-f(a)\in\cE(a,d)\setminus\cE_{\mathrm{sc}}(a,d)$ have closure dimension
below $2^n$ by \cref{thm:cell-generic}. This proves:

\begin{theorem}[Generic benign degeneracy]\label{thm:main}
There exists an open dense semialgebraic set $\Ggen\subseteq\cP_n$ (so its
complement has measure zero) such that every $f\in\Ggen$ has a
unique Boolean minimizer, $d_\star(f)\le n$ is well defined, and at every
order $d_\star(f)\le d\le n$ the moment relaxation has the unique optimum
$b_d(x_\star)b_d(x_\star)^\top$ and an optimal certificate of rank $N_d-1$.
Consequently, for every such $f$, every such order $d$, and every fixed
admissible parameter choice, the conclusions of
\cref{thm:boolean-fixed,cor:admm,cor:pdhg} hold: the three splitting
methods converge eventually linearly from every initialization.
\end{theorem}
\begin{proof}
\Cref{prop:simultaneous} constructs $\Ggen$; for $f\in\Ggen$ and
$d\in[d_\star(f),n]$ the hypotheses of \cref{thm:boolean-fixed} hold and
\cref{sec:algorithms} applies, for every $\sigma>0$, since the argument of
\cref{sec:boolean} gives a positive modulus for each $\sigma$.
\end{proof}

%% file: sections/discussion.tex
\section{Conclusion}
\label{sec:discussion}

We proved that the degeneracy of exact Boolean Moment--SOS relaxations, in
the quotient form of \cref{sec:degeneracy}, is benign for first-order
splitting methods, answering for that formulation the open question
of~\cite{kang25arxiv-admmsdp-linearconv}. Under unique rank-one exactness and one strictly
complementary certificate, the abstract rigidity theorem yields a setwise
error bound around the optimal set. The three splitting methods therefore
converge eventually linearly from every initialization, whichever
certificate they select. Generically, the hypotheses hold at every
uniquely rank-one exact order.

\paragraph{Limitations.}
Our guarantee concerns the algorithmic residual, not objective-function
growth. It is a convergence-rate statement rather than a complexity
theorem, with instance-dependent constants, no bound on the entrance
time, and exact algorithms with fixed scalar parameters only. The rigidity abstract theorem removes
primal uniqueness; yet the Boolean theorem cannot (\cref{rem:tied-objectives}). The verification is tied to the Boolean
quotient through the two identities of \cref{lem:boolean-geometry}.
Redundant encodings (\cref{fig:motivation}(a)), sparse,
symmetry-reduced, or constrained hierarchies change the quotient algebra,
and the propagation can
fail outside it (\cref{ex:not-abstract}; \cite{kang26arxiv-afterimage}).

\paragraph{Open problems.}
First, a \emph{quantitative} rigidity: our constants come from
compactness; we conjecture that for the Boolean class they admit explicit
instance-parameter bounds, the missing step from rate separation to
complexity separation. Second, \emph{beyond
the Boolean quotient}: which exact relaxations satisfy the two kernel
identities of \cref{thm:rigidity}? We believe residual-field
rigidity is the right abstraction for exact relaxations with degenerate
optimal sets, and we hope it extends beyond the Boolean cube.

%% file: sections/app_related.tex
\section{Related Work}
\label{app:related}

\Cref{sec:intro} cites prior work in context; this appendix organizes it
into four threads.

\paragraph{Pseudo-Boolean optimization.}
PBO is a classical umbrella for discrete
optimization~\cite{boros02dam-pbo,kochenberger14joco-ubqp}. Two classical
paradigms use semidefinite relaxations: exact solvers that combine
branch-and-bound with polyhedral or SDP bounds~\cite{rendl10mp-biqmac}, and
approximation algorithms built on SDP
rounding~\cite{goemans95jacm-maxcut}. When these approaches use an SDP relaxation, our result concerns the
cost of \emph{solving} it rather than the cost of rounding or branching.

\paragraph{Moment--SOS for PBO.}
The hierarchy and its $0$--$1$ specializations are classical
\cite{lasserre01siopt-global,parrilo03mp-semidefinite,laurent03mor-comparison};
books and surveys
include~\cite{lasserre09book-moments,blekherman12book-semidefinite,barak14icm-sosproof,moitra20chapter-sostcs}.
Finite convergence and solution extraction are well studied: on the finite
cube full order is always exact, generic problems enjoy finite
convergence~\cite{nie14mp-optimality}, and flatness conditions certify
exactness and permit extraction~\cite{curto05jot-truncated}. On the
negative side, the worst-case lower bounds cited in \cref{sec:intro}
\cite{laurent03mor-cutpolytope,grigoriev01cc-knapsack,schoenebeck08focs-lasserre}
show that no fixed order suffices for all instances, and sum-of-squares
lower bounds extend to average-case problems such as planted
clique~\cite{barak19sic-sos-plantedclique}. The empirical
low-order-exactness literature
\cite{habibi25jgo-lasserre-ubqp,wang25jsc-manisdp,kang25arxiv-admmsdp-linearconv,yang23mp-stride} motivates our
conditional-on-exactness viewpoint. To our knowledge, the geometry of the
\emph{exact} regime at orders $d\ge2$---unique rank-one moment optimum,
primal degeneracy, a generically positive-dimensional certificate set---has
not been used to analyze first-order solvers.

\paragraph{Error bounds and other solvers for degenerate SDP.}
The regularity picture for splitting methods was summarized in
\cref{sec:intro}. Other solvers rest on different conditions. Augmented
Lagrangian methods converge linearly in their outer loop once a strictly
complementary pair \emph{exists} and the subproblems are solved accurately
enough~\cite{rockafellar76mor-augmented-lagrangians,rockafellar76sicon-monotone,liao24neurips-inexact-alm}.
The reason is quadratic growth of the primal and dual objectives, which for
SDP follows from one such pair with compact solution
sets~\cite{liao24neurips-inexact-alm}. Second-order growth is also the
standing assumption of the multiple-solution ALM analysis
of~\cite{cui16arxiv-superlinear-alm-sdp}, and growth of a convex objective
is equivalent to the proximal error bound and to linear rates of proximal
methods~\cite{drusvyatskiy18mor-error-bounds-quadratic-growth}. The cost
then moves into the subproblems. Generic first-order inner solvers are
again worst-case sublinear, and the fast low-rank factorized
route~\cite{burer03mp-nonlinear} needs a unique, strictly complementary
primal optimum (``primal simplicity'', see below) and a local
initialization~\cite{ding25arxiv-new-understanding-sdp-lowrank}. Spectral
bundle methods~\cite{helmberg00siopt-spectral-bundle} are linear under
strict complementarity and sufficient \emph{rank capture}; the capture
requirement is set by the kernel of the limiting dual slack, which can be
large for degenerate SDPs even when the primal optimum has rank
one~\cite{ding23siopt-spectral-bundle-revisit}. For semidefinite
feasibility sets without Slater-type regularity, the general guarantee is
a H\"older bound with exponent tied to the singularity
degree~\cite{sturm00siopt-error-bound-lmi}; see
also~\cite{zhou17mp-unified-error-bounds} for nuclear-norm problems.
Splitting methods are governed by a different quantity, the
Douglas--Rachford fixed-point residual. An error bound on this residual
gives linear convergence of Douglas--Rachford and ADMM, and it is known for
piecewise linear-quadratic and well-posed conic feasibility
problems~\cite{pena21arxiv-dr-error-bound,aspelmeier16siims-admm-dr-local-linear}.
For SDP, objective growth does not imply it: the tied-objective instance of
\cref{app:tied-objectives} meets the augmented Lagrangian hypotheses, yet
its Douglas--Rachford residual bound fails. Complementarity rigidity is
exactly what separates the two cases (\cref{thm:rigidity}).

\paragraph{Setwise versus pointwise guarantees.}
Our result differs from the pointwise analyses cited in \cref{sec:intro}
in two ways. First, the guarantee is \emph{setwise}, so the endpoint the
algorithm selects need satisfy neither strict complementarity nor
nondegeneracy. Second, the regularity is \emph{imported}, from a
uniform interior modulus to every singular boundary endpoint. The bound
\eqref{eq:canonical-eb} is a linear growth condition on the DR residual
made uniform over the fixed-point set, and by \cref{thm:rigidity} it is
equivalent to the stall cone equaling the tangent cone at every endpoint.
The companion paper~\cite{kang26arxiv-admmsdp-limitdyn} studies the
failure of this equality at a selected endpoint. There the gap
$\cC(\bar Z)\setminus\Tan_{\Zopt}(\bar Z)$ consists of the two cross
blocks coupling $\ran\bar X$ with $K_{\bar S}$ and $K_{\bar X}$ with
$\ran\bar S$ (its Proposition~3), and each direction in it carries a
nonzero second-order drift (its Proposition~6). Linear growth at $\bar Z$
closes the gap (its Corollary~1(3)). The identities \eqref{eq:kernel-tests}
say that these cross blocks are empty at every relative-boundary factor,
and (iii)$\Leftrightarrow$(v) is the setwise converse of that corollary.
The one-sided case, a unique primal optimum plus one strictly
complementary pair, is the \emph{primal simplicity} condition
of~\cite[Definition~5]{ding2021siopt-simplicity-lowrank-sdp}; the
two-sided theorem drops the uniqueness. In general conic programs, ADMM,
sGS--ADMM, and PDHG can become arbitrarily slow near degenerate
optima~\cite{kang26arxiv-afterimage,jiang26arxiv-linearconv-pdhg-sdp};
exact Boolean relaxations exclude this near their optimal set, with
instance-dependent constants (\cref{sec:discussion}).

%% file: sections/app_formulation.tex
\section{The Boolean Moment--SOS Formulation}
\label{app:formulation}

Throughout this appendix, fix $1\le d\le n$ and put $N=N_d$; recall the
index sets \eqref{eq:index-sets}, the XOR-class matrices $A_\Gamma$, and the
moment matrix \eqref{eq:moment-matrix}. As in \cref{sec:boolean}, set
\begin{equation}\label{eq:boolean-spaces}
 \cH_d:=\spanof\{A_\Gamma:\Gamma\in\cI_{2d}\},\qquad
 \cT_d:=\spanof\{A_\Gamma:\emptyset\ne\Gamma\in\cI_{2d}\},\qquad
 \cR_d:=\cT_d^\perp .
\end{equation}
Every $\Gamma\in\cI_{2d}$ occurs as a symmetric difference of two sets in
$\cI_d$: split it into two parts of size at most $d$. The XOR class of
$\Gamma$ is the set of unordered pairs $\{A,B\}$ with $A,B\in\cI_d$ and
$A\triangle B=\Gamma$, i.e.\ the support of $A_\Gamma$. Distinct classes
are disjoint, so the matrices $A_\Gamma$ are Frobenius orthogonal, and
$A_\emptyset=I_N$. Hence $\cH_d=\spanof\{I_N\}\oplus\cT_d$ as in
\eqref{eq:boolean-L}. Every matrix in $\cT_d$ has zero diagonal.

\paragraph{The Gram-to-polynomial map.}
Define the polynomial-valued linear map
\begin{equation}\label{eq:gram-map}
 \cG_d(Y):=\sum_{\Gamma\in\cI_{2d}}\ip{A_\Gamma}Y\,x^\Gamma,
 \qquad
 b_d(x)^\top Yb_d(x)\equiv\cG_d(Y)(x)\pmod{I_{\mathrm B}},
\end{equation}
so $\ker(\cG_d)=\cH_d^\perp$. The canonical representative
$C_f=\sum_\Gamma f_\Gamma A_\Gamma/\normF{A_\Gamma}^2$ of
\cref{sec:degeneracy} satisfies $\cG_d(C_f)=f$ and
$\ip{C_f}{M_d(y)}=\sum_\Gamma f_\Gamma y_\Gamma$ by orthogonality of the
classes. For moments, let
$L_y(\sum_\Gamma h_\Gamma x^\Gamma):=\sum_\Gamma h_\Gamma y_\Gamma$ be the
Riesz functional. For $g=c^\top b_d$ one has
$L_y(g^2\bmod I_{\mathrm B})=c^\top M_d(y)c$. Hence $M_d(y)\succeq0$ if and
only if $L_y(g^2\bmod I_{\mathrm B})\ge0$ for every square-free $g$ of
degree at most $d$.

\paragraph{The independent equality encoding.}
Let $E_{A,B}$ ($A,B\in\cI_d$) be the matrix units, and define the
symmetric units $F_{A,B}:=E_{A,A}$ if $A=B$ and
$F_{A,B}:=\tfrac12(E_{A,B}+E_{B,A})$ otherwise. Then
$\ip{F_{A,B}}X=X_{A,B}$ and $\cG_d(F_{A,B})=x^{A\triangle B}$, and the
$F_{A,B}$ with $A\le B$ (in a fixed total order on $\cI_d$) form a basis
of $\Sym^N$. For each XOR class $\Gamma$ choose one reference pair
$\{A,B\}$ (with $A=B=\emptyset$ for $\Gamma=\emptyset$), and write
$F^\circ_\Gamma:=F_{A,B}$ for its symmetric unit. For every other pair
$\{A,B\}$ in the class, include the consistency matrix
$D_{A,B}:=F_{A,B}-F^\circ_\Gamma$, where $\Gamma=A\triangle B$. Listing
these as $D_1,\ldots,D_{m-1}$, define
\begin{equation}\label{eq:Ad}
 \cA_dX:=\bigl(X_{\emptyset,\emptyset},\ip{D_1}X,\ldots,\ip{D_{m-1}}X\bigr),
 \qquad b:=e_1,
 \qquad m=\frac{N(N+1)}2-|\cI_{2d}|+1 .
\end{equation}

\begin{proposition}[Exact standard SDP encoding]\label{prop:encoding}
The matrices $D_1,\ldots,D_{m-1}$ form a basis of $\ker(\cG_d)=\cH_d^\perp$;
the operator $\cA_d$ has full row rank; and
\begin{equation}\label{eq:encoding-identities}
 \{X:\cA_dX=e_1\}=I_N+\cT_d,\qquad
 \ker(\cA_d)=\cT_d,\qquad
 \ran(\cA_d^*)=\cR_d,\qquad
 \cA_d^\dagger e_1=I_N .
\end{equation}
\end{proposition}
\begin{proof}
The $F_{A,B}$ with $A\le B$ form a basis of $\Sym^N$. Within one XOR
class, $\cG_d$ maps each basis element to the same monomial. Hence its
kernel on that class consists of the coefficient vectors summing to zero,
and the differences from the reference unit $F^\circ_\Gamma$ form a basis
of it. Distinct classes have disjoint supports. This proves the kernel
assertion and independence of the $D_i$.
Since $\cG_d(E_{\emptyset,\emptyset})=1$, the normalization row is not in
their span, so $\cA_d$ has full row rank. The equalities $\cA_dX=e_1$ state
that the entries of $X$ are constant on each XOR class and that the common
diagonal value is one. This is the first identity in
\eqref{eq:encoding-identities}. Its direction space is $\cT_d$, and
orthogonal complementation gives the row space. Finally
$I_N\in\cR_d$ (matrices in $\cT_d$ have zero diagonal) and $\cA_dI_N=e_1$,
so $I_N$ is the unique feasible point in $\ran(\cA_d^*)$, namely
$\cA_d^\dagger e_1$.
\end{proof}

One may instead impose all pairwise consistency equalities, but then the
rows are redundant and $\cA_d\cA_d^*$ is singular; the independent choice
\eqref{eq:Ad} is a linear-algebra preprocessing decision, unrelated to
strict complementarity.

\paragraph{The SOS problem is the exact dual.}
With the encoding \eqref{eq:Ad}, the primal is
$\min\{\ip{C_f}X:\cA_dX=e_1,\ X\succeq0\}$. It has the strictly feasible
point $I_N$, and its feasible set is compact, since the diagonal entries
are one and PSD forces $|X_{A,B}|\le1$. Hence the minimum is attained and
strong duality holds. Writing the equality multiplier as
$\lambda=(\gamma,\nu_1,\ldots,\nu_{m-1})$, the Lagrangian dual is
\begin{equation}\label{eq:matrix-dual}
 \max_{\gamma,\nu,S}\ \gamma
 \quad\text{s.t.}\quad
 \gamma E_{\emptyset,\emptyset}+\sum_{j=1}^{m-1}\nu_jD_j+S=C_f,
 \qquad S\succeq0 .
\end{equation}
Applying $\cG_d$ eliminates the consistency multipliers and gives the SOS
problem
\begin{equation}\label{eq:sos-dual}
 (\mathrm D_d)\qquad
 \max_{\gamma\in\R,\ S\succeq0}\ \gamma
 \quad\text{s.t.}\quad
 f-\gamma\equiv b_d(x)^\top Sb_d(x)\pmod{I_{\mathrm B}} .
\end{equation}
No information is lost in this elimination: any feasible
$(\gamma,S)$ of \eqref{eq:sos-dual} has
$C_f-\gamma E_{\emptyset,\emptyset}-S\in\ker(\cG_d)$, so
\cref{prop:encoding} reconstructs the unique multipliers $\nu_j$ in
\eqref{eq:matrix-dual}. The SOS Gram matrix $S$ is therefore the standard
SDP dual slack. At optimality of an exact relaxation, the optimal slack set
is \eqref{eq:slack-set}.

\begin{lemma}[Trace identity and compactness]\label{lem:slack-trace}
Suppose the relaxation at order $d$ is exact with value $f_\star$. Every
$S\in\Sopt$ satisfies $\tr S=f_\emptyset-f_\star$; consequently $\Sopt$ is
compact.
\end{lemma}
\begin{proof}
$\tr S=\ip{A_\emptyset}S$ is the constant coefficient of $\cG_d(S)=f-f_\star$,
which is $f_\emptyset-f_\star$. The set is closed, and
$\normF S\le\tr S$ for PSD matrices, so it is bounded.
\end{proof}

\paragraph{Primal degeneracy.}
We now prove \cref{prop:degeneracy}. Throughout, $d\ge d_\star(f)$, so
$X_\star=uu^\top=Nqq^\top$ is the unique primal optimum
(\cref{sec:degeneracy}). Primal nondegeneracy of
$X_\star$~\cite{alizadeh9mp-complementarity-nondegeneracy} means that
$\cA_d$ is surjective on the tangent space at $X_\star$ to the manifold of
rank-one PSD matrices. This tangent space is
\begin{equation}\label{eq:nondeg-definition}
 \{qv^\top+vq^\top:v\in\R^N\}=\Sym(V)^\perp,
 \qquad\text{so nondegeneracy reads}\qquad
 \cA_d\bigl(\Sym(V)^\perp\bigr)=\R^m .
\end{equation}
This is a condition on the optimum relative to the constraints, distinct
from mere singularity of $X_\star$.

\begin{proof}[Proof of \cref{prop:degeneracy}]
By the constant-diagonal identity \eqref{eq:transverse-intersection},
$\ker(\cA_d)\cap\Sym(V)^\perp=\cT_d\cap\Sym(V)^\perp=\{0\}$, so $\cA_d$ is
injective on the $N$-dimensional tangent space
\eqref{eq:nondeg-definition}; hence nondegeneracy holds if and only if
$m=N$.
For the dimension formula, $\cR_d+\Sym(V)=\Sym^N$ because the orthogonal
complement of the sum is $\cT_d\cap\Sym(V)^\perp=\{0\}$; therefore
\[
 \dim \cL=\dim\bigl(\cR_d\cap\Sym(V)\bigr)
 =m+\frac{(N-1)N}2-\frac{N(N+1)}2=m-N,
\]
which is the dimension formula of \cref{prop:degeneracy} by \eqref{eq:Ad}. At $d=1$, $N=n+1$ and
$|\cI_2|=1+n+\binom n2$, so
$m=\frac{(n+1)(n+2)}2-\binom n2-n-1+1=n+1=N$. For $d\ge2$, take the two
Gram vectors $w_1=e_\emptyset-a\,e_{\{i,j\}}$ and
$w_2=e_{\{i\}}-a\,e_{\{j\}}$ of \cref{sec:degeneracy}. Both are orthogonal
to $u$, and they represent the same polynomial \eqref{eq:same-poly}. Hence
$w_1w_1^\top-w_2w_2^\top$ is a nonzero element of
$\Sym(V)\cap\cH_d^\perp=\cL$ by \eqref{eq:boolean-L}, so $\dim\cL\ge1$ and
$m>N$. For the final clause, \cref{ass:two-sided} holds here:
$\Xopt=\{X_\star\}$, $\Sopt$ is compact (\cref{lem:slack-trace}), and
$(X_\star,\Ssc)$ is an anchor pair. So \cref{lem:product-geometry} gives
$\Sopt=(\Ssc+\cL)\cap\Sym_+(V)$. Since $\Ssc$ is positive definite on $V$,
a relative neighborhood of $\Ssc$ in $\Ssc+\cL$ stays inside $\Sopt$.
Hence $\aff\Sopt=\Ssc+\cL$ and $\dim\Sopt=\dim\cL=m-N$.
\end{proof}

Finally, assume strict complementarity. Then $\Sopt$ is compact and of
positive dimension, so its relative boundary is nonempty. Every certificate
positive definite on $V$ lies in $\ri\Sopt$ (\cref{lem:product-geometry}).
Hence every relative-boundary certificate loses rank, as claimed in
\cref{sec:degeneracy}. Without strict complementarity, the
positive-dimension conclusion can genuinely fail. The affine set $S_0+\cL$
through a certificate $S_0$ can meet $\Sym_+(V)$ in the single point
$S_0$, which is then rank-deficient. This is why \cref{prop:degeneracy}
states the conclusion as a conditional clause.

%% file: sections/app_spectral.tex
\section{Endpoint Spectral Calculus}
\label{app:spectral}

\paragraph{Conventions.}
If $E\subseteq\R^p$ is a subspace with orthogonal projector $P_E$, then
$\Sym(E):=\{H\in\Sym^p:H=P_EHP_E\}$ and
$\Sym_+(E):=\Sym(E)\cap\Sym_+^p$; the complement $\Sym(E)^\perp$ is taken in
$\Sym^p$. For a nonempty closed convex set $C$ and $z\in C$,
\[
 T_C(z):=\{h:\exists t_k\downarrow0,\ z_k\in C,\ (z_k-z)/t_k\to h\},
 \qquad
 N_C(z):=\{h:\ip h{w-z}\le0\ \forall w\in C\},
\]
and $N_C(z):=\emptyset$ for $z\notin C$; $\ri$, $\rb$, $\aff$, and $\cl$
denote relative interior, relative boundary, affine hull, and closure.
Vectors take the Euclidean norm $\normtwo{\cdot}$, and $\normop{\cdot}$ is
the operator norm. Directional derivatives are
$F'(z;h):=\lim_{t\downarrow0}(F(z+th)-F(z))/t$, positively homogeneous in
$h$ but not necessarily linear. For a linear map $B:E_1\to E_2$ between
Euclidean spaces with $E_1\ne\{0\}$, set
$\sigma_{\min}(B):=\min_{\|h\|=1}\|Bh\|$; it is positive if and only if $B$ is
injective, and
$|\sigma_{\min}(B_1)-\sigma_{\min}(B_2)|\le\normop{B_1-B_2}$.

\paragraph{Blocks.}
Throughout, $U=\ran\Xsc$ and $V=\ran\Ssc$ are the fixed spaces of
\cref{sec:rigidity:setup}, with fixed orthonormal bases $B_U,B_V$. For
$H\in\Sym^p$ we write $H_{UU}:=B_U^\top HB_U$, $H_{UV}:=B_U^\top HB_V$,
$H_{VV}:=B_V^\top HB_V$; a statement such as ``$H_{VV}=0$'' means
$P_VHP_V=0$ and does not depend on the bases. For a matrix $A$ supported
on $U$ we also write $A|_U:=A_{UU}$, and for one supported on $V$ we write
$A|_V:=A_{VV}$. Thus $\lambda_{\min}(\Xsc|_U)$ is the smallest eigenvalue
of the $UU$ block of $\Xsc$. For an endpoint
$\bar Z=\bar X-\sigma\bar S$ the moving spaces are $K_{\bar X}=\ker\bar
X\cap U$ and $K_{\bar S}=\ker\bar S\cap V$, so that
$U=\ran\bar X\oplus K_{\bar X}$ and $V=K_{\bar S}\oplus\ran\bar S$. (In
the block notation of~\cite{kang26arxiv-admmsdp-limitdyn}, these four
spaces are $\alpha_+,\alpha_0^{\mathrm P},\alpha_0^{\mathrm D},\alpha_-$.)

\subsection{The DR encoding and the optimal product}
\label{app:drs-slice}

\begin{lemma}[DR and KKT encoding]\label{lem:drs}
$T_{\sigma,1}$ is the Douglas--Rachford map
$Z\mapsto Z+\operatorname{prox}_{\sigma F}(2\Pi_+(Z)-Z)-\Pi_+(Z)$ for
$F(X)=\ip CX+\iota_{\{\cA X=b\}}(X)$ and $G=\iota_{\Sym_+^p}$; each
step applies the PSD projection $\Pi_+=\operatorname{prox}_{\sigma G}$
first and the affine proximal map $\operatorname{prox}_{\sigma F}$ second. Its fixed
points are exactly the matrices $X-\sigma S$ with $(X,S)$ a KKT pair,
decoded via $X=\Pi_+(Z)$, $S=-\sigma^{-1}\Pi_-(Z)$.
\end{lemma}

\begin{proof}[Proof of \cref{lem:drs}]
The affine proximal map is
$\operatorname{prox}_{\sigma F}(W)=X_b+\Pnull(W-\sigma C)$. Substituting it
into $Z^+=Z+\operatorname{prox}_{\sigma F}(2\Pi_+(Z)-Z)-\Pi_+(Z)$ gives
\eqref{eq:res-components}; this is the standard DR
construction~\cite{eckstein92mp-drs,bauschke17springer-convex-analysis-hilbert-spaces}.
The Moreau decomposition of $Z$ gives $X:=\Pi_+(Z)\succeq0$,
$S:=-\sigma^{-1}\Pi_-(Z)\succeq0$, and $XS=0$. The two orthogonal equations
$\rpar=\rper_\sigma=0$ are equivalent to $\cA X=b$ and
$C-S\in\ran(\cA^*)$, hence to KKT feasibility plus complementarity;
surjectivity of $\cA$ makes the multiplier unique. Conversely any KKT pair
$(X,S)$ gives the fixed point $Z=X-\sigma S$.
\end{proof}

\begin{lemma}[Optimal product geometry]\label{lem:product-geometry}
Under \cref{ass:two-sided},
\begin{equation}\label{eq:product-app}
 \Xopt=(\Xsc+\cM)\cap\Sym_+(U),\qquad
 \Sopt=(\Ssc+\cL)\cap\Sym_+(V),\qquad
 \Zopt=\Fix T_{\sigma,1}=\Xopt-\sigma\Sopt .
\end{equation}
Each $\bar Z\in\Zopt$ has a unique representation $\bar X-\sigma\bar S$
with $(\bar X,\bar S)\in\Xopt\times\Sopt$. Moreover
$\spanof(\Xopt-\Xopt)=\cM$, $\spanof(\Sopt-\Sopt)=\cL$,
$\aff\Zopt=\Xsc-\sigma\Ssc+\cD$, and
\begin{equation}\label{eq:ri-rb}
 \ri\Xopt=\{X\in\Xopt:X_{UU}\succ0\},\qquad
 \ri\Sopt=\{S\in\Sopt:S_{VV}\succ0\},
\end{equation}
with the relative boundaries the complementary sets; $\bar X-\sigma\bar S\in
\ri\Zopt$ if and only if $\bar X\in\ri\Xopt$ and $\bar S\in\ri\Sopt$.
Finally, $\Zopt$ is compact and convex.
\end{lemma}

\begin{proof}[Proof of \cref{lem:product-geometry}]
For any primal optimizer $X$ and optimal slack $S$, equality of the optimal
values gives $\ip XS=0$, hence $XS=SX=0$ since both are PSD. Taking
$S=\Ssc$ shows $\ran X\subseteq\ker\Ssc=U$, and taking $X=\Xsc$ shows
$\ran S\subseteq\ker\Xsc=V$. Primal feasibility gives $X-\Xsc\in\cN$ and
dual feasibility gives $S-\Ssc\in\cR$, which proves the inclusions
``$\subseteq$'' in the first two identities of \eqref{eq:product-app}.
Conversely, $X\in(\Xsc+\cM)\cap\Sym_+(U)$ is feasible and complementary to
$\Ssc$, hence optimal; symmetrically for $S$. Since $\Xsc$ is positive
definite on $U$, $\Xsc\pm\eps M$ is PSD for every $M\in\cM$ and small
$\eps>0$, so $\spanof(\Xopt-\Xopt)=\cM$; likewise for $\cL$.

Every pair $(X,S)\in\Xopt\times\Sopt$ has disjoint supports, so $XS=0$,
$\Pi_+(X-\sigma S)=X$, $\Pi_-(X-\sigma S)=-\sigma S$, and
$\delta_\sigma(X-\sigma S)=0$ by \cref{lem:drs}; conversely a fixed point
decodes to a KKT pair, which is optimal. This proves the third identity.
Compressing $\bar Z$ to $U$ and $V$ recovers $\bar X$ and $-\sigma\bar S$,
so the representation is unique. The map $(X,S)\mapsto X-\sigma S$ is a
linear isomorphism from $\Sym(U)\times\Sym(V)$ onto
$\Sym(U)\oplus\Sym(V)$. This gives the affine hull and reduces the
relative-interior statement to the factors. A point of $\Xopt$ positive
definite on $U$ has a relative neighborhood inside $\Xopt$. A point
$\bar X$ with singular $\bar X_{UU}$ is not relatively interior: the line
through $\bar X$ and $\Xsc$, continued past $\bar X$, leaves the PSD cone
immediately. The same argument applies to $\Sopt$ inside $\Sym_+(V)$,
which proves \eqref{eq:ri-rb}. Compactness and convexity of $\Zopt$ follow
from those of the factors.
\end{proof}

\subsection{The endpoint derivative formula}
\label{app:deriv-validity}

Fix an endpoint $\bar Z=\bar X-\sigma\bar S$. Choose orthonormal
eigenvector matrices $Q_+$ of $\bar X$ on $\ran\bar X$ (eigenvalues
$\xi_i>0$) and $Q_-$ of $\bar S$ on $\ran\bar S$ (eigenvalues $s_j>0$).
Choose also an orthonormal basis $Q_K$ of the zero eigenspace
$K:=K_{\bar X}\oplus K_{\bar S}$ of $\bar Z$. The frame
$Q:=[\,Q_+\ Q_K\ Q_-\,]$ is orthogonal, and
\begin{equation}\label{eq:eigen-chart}
 Q^\top\bar ZQ=\diag(\Lambda_+,\,0,\,-\sigma\Lambda_-),\qquad
 \widetilde H:=Q^\top HQ,\qquad
 \Lambda_+=\diag(\xi_i)\succ0,\quad\Lambda_-=\diag(s_j)\succ0 .
\end{equation}
In this chart the directional derivative of the PSD projection takes the
classical form~\cite{sun02mor-semismooth-matrix-valued-function}
\begin{equation}\label{eq:positive-deriv}
 Q^\top\,\Pi_+'(\bar Z;H)\,Q
 =\begin{bmatrix}
 \widetilde H_{++}&\widetilde H_{+K}&\Omega\circ\widetilde H_{+-}\\
 \widetilde H_{K+}&\Pi_+(\widetilde H_{KK})&0\\
 (\Omega\circ\widetilde H_{+-})^\top&0&0
 \end{bmatrix},\qquad
 \Omega_{ij}:=\frac{\xi_i}{\xi_i+\sigma s_j}\in(0,1).
\end{equation}
Here $\Pi_-'(\bar Z;H)=H-\Pi_+'(\bar Z;H)$ and $\circ$ is the entrywise
product. The entries of $\Omega$ are the divided differences of
$t\mapsto\max\{t,0\}$ across the nonzero spectral blocks. The $+K$ block
carries the divided difference $\xi_i/(\xi_i-0)=1$, and the $K-$ block
carries $0/(0+\sigma s_j)=0$. The zero block contributes its own
directional derivative $\Pi_+(\widetilde H_{KK})$. The formula is valid with repeated
eigenvalues, because on any pair of repeated-eigenvalue eigenspaces the
multiplier is constant, so changes of eigenbasis give the same ambient
derivative after conjugation by $Q$. For $E:=\Pi_+'(\bar Z;H)$ and
$F:=\Pi_-'(\bar Z;H)$ we read off
\begin{equation}\label{eq:EF-identities}
 E+F=H,\quad \ip E{\bar S}=0,\quad \ip F{\bar X}=0,\quad
 \ip EF=2\!\!\sum_{i\in+,\,j\in-}\!\!\Omega_{ij}(1-\Omega_{ij})
 \,|(\widetilde H_{+-})_{ij}|^2\ \ge0,
\end{equation}
since $E$ has zero $--$ block, $F$ has zero $++$ block, the $KK$ parts pair
as $\ip{\Pi_+(M)}{\Pi_-(M)}=0$, and only the $+-$ overlap survives. Also
$\normF E,\normF F\le\normF H$, because $\Pi_\pm$ are nonexpansive and a
directional derivative is a limit of difference quotients, and
\begin{equation}\label{eq:res-deriv}
 \delta_\sigma'(\bar Z;H)=-\Prow E-\Pnull F,\qquad
 \normF{\delta_\sigma'(\bar Z;H)}^2=\normF{\Prow E}^2+\normF{\Pnull F}^2 .
\end{equation}

\paragraph{Three kinds of endpoints in the fixed coarse coordinates.}
Whenever $K_{\bar X}=\{0\}$ or $K_{\bar S}=\{0\}$, the same derivative
reads, in the fixed bases $B_U,B_V$,
\begin{equation}\label{eq:coarse-derivative}
 [B_U\ B_V]^\top\,\Pi_+'(\bar Z;H)\,[B_U\ B_V]
 =\begin{bmatrix}H_{UU}^{\sharp}&Y\\Y^\top&H_{VV}^{\sharp}\end{bmatrix},
 \qquad
 \bar X_{UU}\,Y+\sigma\,Y\,\bar S_{VV}=\bar X_{UU}\,H_{UV}.
\end{equation}
We call this $U\oplus V$ block description \emph{coarse}, in contrast with
the eigen-chart \eqref{eq:eigen-chart}, which refines $U$ and $V$ by the
moving kernels. Here $H_{UU}^\sharp$ is $H_{UU}$ with its
$K_{\bar X}K_{\bar X}$ block replaced by the positive part of that block.
Similarly, $H_{VV}^\sharp$ is the positive part of the
$K_{\bar S}K_{\bar S}$ block of $H_{VV}$, embedded in $V$. The cross block
$Y$ is the unique solution of the displayed Sylvester equation. Indeed, in
eigencoordinates the equation reads
$(\xi_i+\sigma s_j)\widetilde Y_{ij}=\xi_i\widetilde H_{ij}$ with
$\xi_i\ge0$ on $U$ and $s_j\ge0$ on $V$. This reproduces the $+-$, $+K$,
and $K-$ blocks of \eqref{eq:positive-deriv}. It is uniquely solvable
exactly when $\xi_i+\sigma s_j>0$ for all pairs, that is, when
$K_{\bar X}=\{0\}$ or $K_{\bar S}=\{0\}$. When both kernels are
nontrivial, the $K_{\bar X}K_{\bar S}$ part of $\Pi_+(\widetilde H_{KK})$
has no such description, and we never use the coarse form there. Three
cases are used:
\begin{enumerate}[label=\textup{(\alph*)},leftmargin=2.2em,itemsep=1pt]
\item \emph{Interior endpoints} ($\bar X_{UU}\succ0$, $\bar S_{VV}\succ0$):
$K=\{0\}$, $\Pi_+$ is Fr\'echet differentiable at $\bar Z$, and
$J_{\bar X,\bar S}:=\Pi_+'(\bar Z;\cdot)$ is the linear map
$H\mapsto\bigl[\begin{smallmatrix}H_{UU}&Y\\Y^\top&0\end{smallmatrix}\bigr]$
of \eqref{eq:sylvester-main}.
\item \emph{Primal-anchored endpoints} $\bar Z=\Xsc-\sigma\bar S$:
$K=K_{\bar S}$, $H^\sharp_{UU}=H_{UU}$, and the Sylvester equation has
$\Xsc|_U\succ0$ on the left.
\item \emph{Dual-anchored endpoints} $\bar Z=\bar X-\sigma\Ssc$:
$K=K_{\bar X}$, $H^\sharp_{VV}=0$, and the Sylvester equation has
$\Ssc|_V\succ0$ on the right.
\end{enumerate}

\subsection{Two uniform Sylvester estimates}
\label{app:sylvester}

\begin{lemma}[Sylvester estimates]\label{lem:sylvester}
Let $A\in\Sym_+(U)$ and $B\in\Sym_+(V)$ be given in the fixed bases, and
let $\mathsf{Syl}:Y\mapsto AY+\sigma YB$ act on $\dim U\times\dim V$
matrices.
\begin{enumerate}[label=\textup{(\alph*)},leftmargin=2.2em,itemsep=1pt]
\item If $A\succeq aI$ with $a>0$, then $\mathsf{Syl}$ is invertible and
$\|\mathsf{Syl}^{-1}\|_{F\to F}\le a^{-1}$, uniformly over all $B\succeq0$.
\item If $B\succeq bI$ with $b>0$, then $\mathsf{Syl}$ is invertible and
$\|\mathsf{Syl}^{-1}\|_{F\to F}\le(\sigma b)^{-1}$, uniformly over all
$A\succeq0$.
\end{enumerate}
In particular, let $a_X:=\lambda_{\min}(\Xsc|_U)$ and
$a_S:=\lambda_{\min}(\Ssc|_V)$. The Sylvester operators of the
primal-anchored endpoints have inverse norm at most $a_X^{-1}$ for every
$\bar S\in\Sopt$. Those of the dual-anchored endpoints have inverse norm
at most $(\sigma a_S)^{-1}$ for every $\bar X\in\Xopt$.
\end{lemma}
\begin{proof}
For every $Y$,
$\ip Y{\mathsf{Syl}Y}=\tr(Y^\top AY)+\sigma\tr(Y^\top YB)$, and both terms
are nonnegative. In case (a) the first term is at least $a\normF Y^2$, in
case (b) the second is at least $\sigma b\normF Y^2$; Cauchy--Schwarz then
gives $\normF{\mathsf{Syl}Y}\ge a\normF Y$, respectively
$\ge\sigma b\normF Y$. Injectivity in equal finite dimensions gives
invertibility and the bounds.
\end{proof}

\subsection{Two-sided tangent and normal cones}
\label{app:cones}

For an endpoint $\bar Z=\bar X-\sigma\bar S$ define the two
kernel-supported PSD cones
$\cK(\bar X):=\{W\in\Sym_+^p:\ran W\subseteq K_{\bar X}\}$ and
$\cK(\bar S):=\{W\in\Sym_+^p:\ran W\subseteq K_{\bar S}\}$.

\begin{proposition}[Endpoint cones]\label{prop:cones}
Under \cref{ass:two-sided}, at $\bar Z=\bar X-\sigma\bar S\in\Zopt$,
\begin{align}
 T_{\Xopt}(\bar X)&=\{M\in\cM:P_{K_{\bar X}}MP_{K_{\bar X}}\succeq0\},&
 N_{\Xopt}(\bar X)&=\cM^\perp-\cK(\bar X),\label{eq:primal-cones}\\
 T_{\Sopt}(\bar S)&=\{\Lambda\in \cL:P_{K_{\bar S}}\Lambda P_{K_{\bar S}}\succeq0\},&
 N_{\Sopt}(\bar S)&=\cL^\perp-\cK(\bar S),\label{eq:dual-cones}\\
 \Tan_{\Zopt}(\bar Z)&=T_{\Xopt}(\bar X)-\sigma T_{\Sopt}(\bar S),&
 N_{\Zopt}(\bar Z)&=\cD^\perp-\cK(\bar X)+\cK(\bar S),\label{eq:shadow-normal}
\end{align}
where the sums are Minkowski sums and all cones are taken in $\Sym^p$.
\end{proposition}
\begin{proof}
For the cone supported on $U$, the tangent and ambient normal formulas are
\[
 T_{\Sym_+(U)}(\bar X)=\{H\in\Sym(U):P_{K_{\bar X}}HP_{K_{\bar X}}\succeq0\},
 \qquad
 N_{\Sym_+(U)}(\bar X)=\Sym(U)^\perp-\cK(\bar X).
\]
The tangent condition is necessary, as one sees by evaluating feasible
secants on vectors of $K_{\bar X}$. It is sufficient for the following
reason. Adding $\eps P_{K_{\bar X}}$ to a direction $H$ makes its kernel
block positive definite. The Schur complement then shows that
$\bar X+t(H+\eps P_{K_{\bar X}})$ is PSD for small $t>0$. Letting
$\eps\downarrow0$ closes the cone. The normal formula is the polar: the
free cross and range blocks force the corresponding normal blocks to
vanish, and the polar of the PSD kernel cone is its negative. For the
slice $\Xopt=(\Xsc+\cM)\cap\Sym_+(U)$ the same approximation can be
performed inside $\cM$. Replace $H$ by $H+\eps(\Xsc-\bar X)$, whose
kernel block is positive definite because $\Xsc\succ0$ on $U$. The
relative Slater point $\Xsc$ justifies the normal-cone intersection
rule~\cite{rockafellar15book-convex-analysis}. Hence
$N_{\Xopt}(\bar X)=\cM^\perp+\Sym(U)^\perp-\cK(\bar X)=\cM^\perp-\cK(\bar X)$,
using $\cM\subseteq\Sym(U)$. This proves \eqref{eq:primal-cones}, and
\eqref{eq:dual-cones} is the same argument on $V$ with the anchor $\Ssc$.

For the product, the map $(X,S)\mapsto X-\sigma S$ is a linear
isomorphism of $\Sym(U)\times\Sym(V)$ onto $\Sym(U)\oplus\Sym(V)$, so it
carries the tangent cone of $\Xopt\times\Sopt$, which is the product of
the tangent cones, onto $T_{\Xopt}(\bar X)-\sigma T_{\Sopt}(\bar S)$. For
the normal cone, $H\in N_{\Zopt}(\bar Z)$ if and only if
$\ip H{X-\bar X}\le0$ for all $X\in\Xopt$ and $\ip H{S-\bar S}\ge0$ for
all $S\in\Sopt$ (take $S=\bar S$ for the first inequality and $X=\bar X$
for the second), that is,
$H\in N_{\Xopt}(\bar X)\cap\bigl(-N_{\Sopt}(\bar S)\bigr)
=(\cM^\perp-\cK(\bar X))\cap(\cL^\perp+\cK(\bar S))$. If
$H=Y-W_U+W_V$ with $Y\in\cD^\perp$, $W_U\in\cK(\bar X)$,
$W_V\in\cK(\bar S)$, then $H$ lies in both sets because
$W_V\in\Sym(V)\subseteq\cM^\perp$ and $W_U\in\Sym(U)\subseteq \cL^\perp$.
Conversely, if $H=Y_1-W_U=Y_2+W_V$ with $Y_1\in\cM^\perp$,
$Y_2\in \cL^\perp$, then $Y:=H+W_U-W_V=Y_1-W_V=Y_2+W_U$ lies in
$\cM^\perp\cap \cL^\perp=\cD^\perp$, so $H=Y-W_U+W_V$. The sign of the
slack cone is positive because the fixed-point variable contains
$-\sigma S$, and positive scalars do not change cones.
\end{proof}

At a relative-interior endpoint both kernels vanish, so
$\Tan_{\Zopt}(\bar Z)=\cD$ and $N_{\Zopt}(\bar Z)=\cD^\perp$. At a boundary
endpoint the two extra cones may be nonzero, and the three summands in
\eqref{eq:shadow-normal} are \emph{not} orthogonal in general. This is the
precise sense in which the boundary is harder than the interior, and why
\cref{lem:normal-bound} bounds the summands through the anchors rather
than through orthogonality.

\subsection{Three PSD matrix bounds}
\label{app:psd-block}

\begin{lemma}[PSD off-block bound]\label{lem:psd-cauchy}
If $S\succeq0$ and $B_E,B_F$ are orthonormal frames for orthogonal
subspaces, then
$\normF{B_E^\top SB_F}^2\le\tr(B_E^\top SB_E)\,\tr(B_F^\top SB_F)$.
\end{lemma}
\begin{proof}
The block is $(S^{1/2}B_E)^\top(S^{1/2}B_F)$, so
$\normF{B_E^\top SB_F}\le\normF{S^{1/2}B_E}\,\normF{S^{1/2}B_F}$; square and
use $\normF{S^{1/2}B_E}^2=\tr(B_E^\top SB_E)$, likewise for $B_F$.
\end{proof}

\begin{lemma}[PSD product bound]\label{lem:psd-product}
For $A,B\in\Sym_+^p$,
$\normF{AB}^2\le\normop A\,\normop B\,\ip AB$.
\end{lemma}
\begin{proof}
Since $A^2\preceq\normop A\,A$ and $B^2\preceq\normop B\,B$, and trace
pairings of PSD matrices preserve the Loewner order,
$\normF{AB}^2=\tr(A^2B^2)\le\normop A\tr(AB^2)\le\normop A\normop B\tr(AB)$.
\end{proof}

\begin{lemma}[Projection lower bound]\label{lem:pv-bound}
Let $W\in\Sym_+^p$ and let $P_E$ be the orthogonal projector onto a
subspace $E\subseteq\R^p$. Then $W\succeq aP_E$ for some $a>0$ if and only
if $E\subseteq\ran(W)$; in that case, when $W\ne0$, one may take $a=\lambda_+(W)$, the
smallest positive eigenvalue of $W$.
\end{lemma}
\begin{proof}
Necessity: if $W\succeq aP_E$ and $z\in\ker(W)$, then
$0=z^\top Wz\ge a\normtwo{P_Ez}^2$, so $P_Ez=0$; hence $\ker(W)\perp E$,
i.e.\ $E\subseteq\ker(W)^\perp=\ran(W)$. Sufficiency: for every $x$,
$x^\top Wx\ge\lambda_+(W)\normtwo{P_{\ran(W)}x}^2$, and $E\subseteq\ran(W)$
gives $\normtwo{P_{\ran(W)}x}\ge\normtwo{P_Ex}$, so
$x^\top Wx\ge\lambda_+(W)\,x^\top P_Ex$.
\end{proof}

\begin{remark}[Positive definiteness on $E$ is not enough]\label{rem:pv}
The condition $E\subseteq\ran(W)$ in \cref{lem:pv-bound} cannot be replaced
by ``$W$ is positive definite on $E$'', which only says
$E\cap\ker(W)=\{0\}$. For
$W=\bigl(\begin{smallmatrix}1&1\\1&1\end{smallmatrix}\bigr)$ and
$E=\spanof\{e_1\}$, the quadratic form $(te_1)^\top W(te_1)=t^2$ is
positive on $E$, yet $\det(W-aP_E)=-a<0$ for every $a>0$, so no inequality
$W\succeq aP_E$ holds. In \cref{lem:normal-bound} the stronger condition is
available for both anchors because $\ran(\Xsc)=U$ and $\ran(\Ssc)=V$.
\end{remark}

%% file: sections/app_rigidity_proofs.tex
\section{Proofs of the Rigidity Theorem}
\label{app:rigidity}

Throughout, \cref{ass:two-sided} is in force, endpoints are written
$\bar Z=\bar X-\sigma\bar S$ with $(\bar X,\bar S)\in\Xopt\times\Sopt$, and
we use the conventions of \cref{app:spectral}. Recall
$a_X=\lambda_{\min}(\Xsc|_U)>0$ and $a_S=\lambda_{\min}(\Ssc|_V)>0$ from
\cref{lem:sylvester}, and set
\begin{equation}\label{eq:anchor-constants}
 R_X:=\max_{X\in\Xopt}\normF{\Xsc-X},\qquad
 R_S:=\max_{S\in\Sopt}\normF{\Ssc-S},
\end{equation}
finite by compactness. The implications of \cref{thm:rigidity} are proved
in the order
(i)$\Rightarrow$(ii) (\cref{app:rigidity-tests}),
(ii)$\Rightarrow$(iii) (\cref{app:rigidity-eb}),
(iii)$\Rightarrow$(iv)$\Rightarrow$(i) and
(iii)$\Rightarrow$(v)$\Rightarrow$(ii) (\cref{app:rigidity-linearization}).
The ``automatic'' clause is \cref{lem:automatic-interior}, and the
$\sigma$-independence clause holds because \eqref{eq:kernel-tests} does
not involve $\sigma$ while the constants in (i), (iii), (iv) do.

\subsection{Interior endpoints: both identities and the linearization are automatic}
\label{app:automatic-interior}

\begin{lemma}[Automatic interior identities]\label{lem:automatic-interior}
For $\bar S\in\ri\Sopt$ the primal identity in \eqref{eq:kernel-tests}
holds, and for $\bar X\in\ri\Xopt$ the dual identity holds. Moreover, for
every $\bar Z\in\ri\Zopt$,
\begin{equation}\label{eq:automatic-interior}
 \cC(\bar Z)=\cD=\Tan_{\Zopt}(\bar Z),
\end{equation}
without any of the conditions (i)--(v).
\end{lemma}
\begin{proof}
Let $\bar S\in\ri\Sopt$, so $\bar S_{VV}\succ0$ by \eqref{eq:ri-rb}. If
$E\in\cN$ satisfies $E_{VV}=0$ and $E\bar S=0$, write its blocks as
$\bigl[\begin{smallmatrix}A&B\\B^\top&0\end{smallmatrix}\bigr]$; the $UV$
block of $E\bar S=0$ reads $B\bar S_{VV}=0$, so $B=0$ and
$E\in\Sym(U)\cap\cN=\cM$. The inclusion of $\cM$ in the left side is
immediate, since elements of $\cM$ are supported on $U$ and annihilate
every matrix supported on $V$. The dual statement is symmetric, with
$\bar X_{UU}\succ0$.

Now let $\bar Z=\bar X-\sigma\bar S\in\ri\Zopt$ and $H\in\cC(\bar Z)$. Put
$E:=J_{\bar X,\bar S}(H)$ and $F:=H-E$. The two terms of
\eqref{eq:res-deriv} lie in the orthogonal spaces $\cR$ and $\cN$, so
$\delta_\sigma'(\bar Z;H)=0$ gives $\Prow E=0$ and $\Pnull F=0$, hence
$E\in\cN$, $F\in\cR$, and $\ip EF=0$. By \eqref{eq:EF-identities} at an
interior endpoint, where every weight $\Omega_{ij}$ lies strictly between
$0$ and $1$, this forces $H_{UV}=0$; then $E=P_UHP_U\in\cM$ and
$F=P_VHP_V\in \cL$ by the coarse form \eqref{eq:coarse-derivative}, so
$H\in\cD$. Conversely, for $M\in\cM$ and $\Lambda\in \cL$ one has
$J_{\bar X,\bar S}(M+\Lambda)=M$ (the cross block is zero and the $VV$
block is killed), so $\delta_\sigma'(\bar Z;M+\Lambda)=-\Prow M-\Pnull
\Lambda=0$. Finally, at a relative-interior point of the convex set
$\Zopt$ the tangent cone is the direction space $\cD$
(\cref{lem:product-geometry}).
\end{proof}

Because $\delta_\sigma'(\bar Z;\cdot)$ is linear at interior endpoints and
vanishes on $\cD$, condition (i) implies the quotient bound
\begin{equation}\label{eq:quotient-bound}
 \normF{\delta_\sigma'(\bar Z;H)}\ \ge\ \mu_{\mathrm{ri}}\,\dist_F(H,\cD)
 \qquad(\bar Z\in\ri\Zopt,\ H\in\Sym^p),
\end{equation}
obtained by applying the definition of $\mu_{\mathrm{ri}}$ to the
component of $H$ orthogonal to $\cD$.

\subsection{Uniformly bounded normal decompositions}
\label{app:normal-bound-proof}

\begin{lemma}[Uniformly bounded normal decompositions]\label{lem:normal-bound}
For every endpoint $\bar Z=\bar X-\sigma\bar S$, every
$H\in N_{\Zopt}(\bar Z)$, and every decomposition
$H=Y-W_U+W_V$ with $Y\in\cD^\perp$, $W_U\in\cK(\bar X)$,
$W_V\in\cK(\bar S)$ as in \eqref{eq:shadow-normal},
\begin{equation}\label{eq:explicit-normal-bound}
 \normF{W_U}\le\frac{R_X}{a_X}\normF H,\qquad
 \normF{W_V}\le\frac{R_S}{a_S}\normF H,\qquad
 \normF Y\le\Bigl(1+\frac{R_X}{a_X}+\frac{R_S}{a_S}\Bigr)\normF H .
\end{equation}
The three summands need not be orthogonal, and the decomposition need
not be unique; the bounds hold for every choice.
\end{lemma}

\begin{proof}
Since $W_U\succeq0$, $\tr W_U\ge\normF{W_U}$. We claim
\[
\begin{aligned}
 a_X\normF{W_U}&\le a_X\tr W_U\le\ip{W_U}{\Xsc}=\ip{W_U}{\Xsc-\bar X}\\
 &=-\ip{H}{\Xsc-\bar X}\le\normF H\,\normF{\Xsc-\bar X}\le R_X\normF H .
\end{aligned}
\]
The second inequality holds because $W_U$ is supported on $U$ and
$\Xsc\succeq a_XP_U$ (\cref{lem:pv-bound}, as $\ran\Xsc=U$). The first
equality uses $\bar XW_U=0$. The second equality uses
$Y\perp\cD\ni\Xsc-\bar X$, and that $W_V$ is supported on $V$ while
$\Xsc-\bar X$ is supported on $U$.
Symmetrically, $\Ssc\succeq a_SP_V$, $\bar SW_V=0$, and
$\Ssc-\bar S\in \cL\subseteq\cD$ give
$a_S\normF{W_V}\le\ip{W_V}{\Ssc-\bar S}=\ip H{\Ssc-\bar S}\le R_S\normF H$;
the sign differs because the fixed-point variable carries $-\sigma S$. The
bound on $Y=H+W_U-W_V$ follows by the triangle inequality.
\end{proof}

\subsection{The interior modulus forces the kernel identities:
(i)\texorpdfstring{$\Rightarrow$}{=>}(ii)}
\label{app:rigidity-tests}

\begin{proof}[Proof of \textup{(i)}$\Rightarrow$\textup{(ii)}]
We prove both identities at \emph{every} optimal factor, which is stronger
than required. The inclusions ``$\supseteq$'' are immediate as in
\cref{lem:automatic-interior}.

\emph{Primal identity.} Fix $\bar S\in\Sopt$ and take $E\in\cN$ with
$E_{VV}=0$ and $E\bar S=0$; write $B:=E_{UV}$, so $B\bar S_{VV}=0$. For
$0<\eps\le1$ set $\bar S_\eps:=(1-\eps)\bar S+\eps\Ssc$ and
$Z_\eps:=\Xsc-\sigma\bar S_\eps$. Then $\bar S_\eps\succeq\eps a_SP_V$ is
positive definite on $V$, so $Z_\eps\in\ri\Zopt$ by
\cref{lem:product-geometry}. Let $Y_\eps$ be the cross block of
$J_{\Xsc,\bar S_\eps}(E)$, which by \eqref{eq:coarse-derivative} solves
$\Xsc|_U\,Y_\eps+\sigma Y_\eps(\bar S_\eps)_{VV}=\Xsc|_U\,B$. Subtracting
the same expression evaluated at $B$ and using $B\bar S_{VV}=0$ gives the
exact identity
\begin{equation}\label{eq:anchored-primal-sylvester}
 \Xsc|_U\,(Y_\eps-B)+\sigma(Y_\eps-B)(\bar S_\eps)_{VV}
 =-\sigma\eps\,B\,\Ssc|_V .
\end{equation}
Pairing with $Y_\eps-B$ and using $(\bar S_\eps)_{VV}\succeq0$ yields
$a_X\normF{Y_\eps-B}^2\le\sigma\eps\normF{Y_\eps-B}\normF{B\,\Ssc|_V}$,
so $\normF{Y_\eps-B}\le(\sigma\eps/a_X)\normF{B\,\Ssc|_V}\to0$ and
$J_{\Xsc,\bar S_\eps}(E)\to E$ (its $UU$ block is $E_{UU}$ and its $VV$
block is zero, as is $E_{VV}$). Since $\Prow E=0$,
\eqref{eq:res-deriv} becomes
$\delta_\sigma'(Z_\eps;E)=(\Pnull-\Prow)\bigl(J_{\Xsc,\bar S_\eps}(E)-E\bigr)$,
and $\Pnull-\Prow$ is an orthogonal reflection, so
$\normF{\delta_\sigma'(Z_\eps;E)}=\normF{J_{\Xsc,\bar S_\eps}(E)-E}\to0$.
The quotient bound \eqref{eq:quotient-bound} forces $\dist_F(E,\cD)=0$.
Writing $E=M+\Lambda$ with $M\in\cM$, $\Lambda\in \cL$, the condition
$E_{VV}=0$ gives $\Lambda=0$, so $E\in\cM$.

\emph{Dual identity.} Fix $\bar X\in\Xopt$ and take $F\in\cR$ with
$F_{UU}=0$ and $\bar XF=0$; write $B:=F_{UV}$, so $\bar X_{UU}B=0$. Set
$X_\eps:=(1-\eps)\bar X+\eps\Xsc$ and $Z_\eps:=X_\eps-\sigma\Ssc\in\ri\Zopt$.
The cross block $Y_\eps$ of $J_{X_\eps,\Ssc}(F)$ solves
$(X_\eps)_{UU}Y_\eps+\sigma Y_\eps\,\Ssc|_V=(X_\eps)_{UU}B=\eps\,\Xsc|_U\,B$,
so \cref{lem:sylvester}(b) gives
$\normF{Y_\eps}\le\eps\normF{\Xsc|_UB}/(\sigma a_S)\to0$, and
$J_{X_\eps,\Ssc}(F)\to0$ because its $UU$ block $F_{UU}$ and its $VV$
block vanish identically. Since $\Pnull F=0$,
$\delta_\sigma'(Z_\eps;F)=-(\Prow-\Pnull)J_{X_\eps,\Ssc}(F)\to0$, and
\eqref{eq:quotient-bound} gives $F\in\cD$; the condition $F_{UU}=0$ kills
the $\cM$ component, so $F\in \cL$.
\end{proof}

\subsection{The kernel identities force the error bound:
(ii)\texorpdfstring{$\Rightarrow$}{=>}(iii)}
\label{app:rigidity-eb}

\begin{proof}[Proof of \textup{(ii)}$\Rightarrow$\textup{(iii)}]
By \cref{lem:automatic-interior}, (ii) gives both identities at every
optimal factor. Suppose \eqref{eq:canonical-eb} fails. Since
$\delta_\sigma$ vanishes on $\Zopt$, there are $Z_k\notin\Zopt$ with
$0<t_k:=\dist_F(Z_k,\Zopt)\le1/k$ and
$\normF{\delta_\sigma(Z_k)}/t_k<1/k$. Let $\bar Z_k=\bar X_k-\sigma\bar
S_k$ be the point of the compact convex set $\Zopt$ nearest to $Z_k$, and
put $H_k:=(Z_k-\bar Z_k)/t_k$, so that $\normF{H_k}=1$,
$H_k\in N_{\Zopt}(\bar Z_k)$, and $Z_k=\bar Z_k+t_kH_k$. Define the
\emph{exact} projections and secants
\begin{equation}\label{eq:secants}
 X_k:=\Pi_+(Z_k),\quad S_k:=-\sigma^{-1}\Pi_-(Z_k),\quad
 E_k:=\frac{X_k-\bar X_k}{t_k},\quad
 F_k:=-\sigma\,\frac{S_k-\bar S_k}{t_k}.
\end{equation}
The Moreau decomposition gives $E_k+F_k=H_k$, and nonexpansiveness of
$\Pi_\pm$ gives $\normF{E_k},\normF{F_k}\le1$. After a subsequence,
$(\bar X_k,\bar S_k,E_k,F_k,H_k)\to(\bar X,\bar S,E,F,H)$ with $E+F=H$ and
$\normF H=1$. These are moving-base secants; no derivative at a limiting
endpoint is used anywhere below.

\emph{Residual components and the energy identity.} Since
$\delta_\sigma(\bar Z_k)=0$, \eqref{eq:res-components} gives
$\delta_\sigma(Z_k)/t_k=-\Prow E_k-\Pnull F_k$, whose two terms are
orthogonal; hence
\begin{equation}\label{eq:small-components}
 \Prow E_k\to0,\qquad \Pnull F_k\to0,\qquad E\in\cN,\qquad F\in\cR,
\end{equation}
and $\ip{E_k}{F_k}=\ip{\Prow E_k}{F_k}+\ip{E_k}{\Pnull F_k}\to0$. Exact
complementarity $X_kS_k=0$ and $\bar X_k\bar S_k=0$, expanded through
\eqref{eq:secants}, yields the energy identity
\begin{equation}\label{eq:secant-energy}
 \ip{E_k}{F_k}
 =\frac{\sigma}{t_k^2}\Bigl(\ip{X_k}{\bar S_k}+\ip{\bar X_k}{S_k}\Bigr)\ \ge0;
\end{equation}
both of its terms are nonnegative, so
\begin{equation}\label{eq:gaps}
 \ip{X_k}{\bar S_k}=o(t_k^2),\qquad \ip{\bar X_k}{S_k}=o(t_k^2).
\end{equation}

\emph{Step 1: the wrong-face diagonal blocks vanish.} Optimal primal
matrices are supported on $U$ and optimal slacks on $V$, so
$(E_k)_{VV}=(X_k)_{VV}/t_k\succeq0$ and
$(F_k)_{UU}=-\sigma(S_k)_{UU}/t_k\preceq0$. Since
$\Ssc-\bar S_k\in \cL\subseteq\cR$ and $\Xsc-\bar X_k\in\cM\subseteq\cN$,
\[
 \ip{E_k}{\Ssc}=\ip{\Prow E_k}{\Ssc-\bar S_k}+\frac{\ip{X_k}{\bar S_k}}{t_k}\to0,
 \qquad
 \ip{F_k}{\Xsc}=\ip{\Pnull F_k}{\Xsc-\bar X_k}-\frac{\sigma\ip{\bar X_k}{S_k}}{t_k}\to0,
\]
by compactness, \eqref{eq:small-components}, and \eqref{eq:gaps}. The
fixed anchors now give
$a_S\normF{(E_k)_{VV}}\le\ip{(E_k)_{VV}}{\Ssc|_V}=\ip{E_k}{\Ssc}\to0$ and
$a_X\normF{(F_k)_{UU}}\le-\ip{(F_k)_{UU}}{\Xsc|_U}=-\ip{F_k}{\Xsc}\to0$,
using $\tr A\ge\normF A$ for PSD $A$; hence
\begin{equation}\label{eq:wrong-blocks}
 (E_k)_{VV}\to0,\qquad (F_k)_{UU}\to0,\qquad E_{VV}=0,\qquad F_{UU}=0 .
\end{equation}

\emph{Step 2: annihilation without spectral gaps.} All of
$X_k,S_k,\bar X_k,\bar S_k$ are bounded. \Cref{lem:psd-product} and
\eqref{eq:gaps} give $\normF{X_k\bar S_k}^2=o(t_k^2)$ and
$\normF{\bar X_kS_k}^2=o(t_k^2)$, and $\bar X_k\bar S_k=0$ turns these
into $E_k\bar S_k=X_k\bar S_k/t_k\to0$ and
$\bar X_kF_k=-\sigma\bar X_kS_k/t_k\to0$. In the limit,
\begin{equation}\label{eq:annihilation}
 E\bar S=0,\qquad \bar XF=0 .
\end{equation}
No lower bound on the positive eigenvalues of $\bar X_k$ or $\bar S_k$ has
been used.

\emph{Step 3: the limits lie in the direction spaces.} By
\eqref{eq:small-components}, \eqref{eq:wrong-blocks}, and
\eqref{eq:annihilation}, $E\in\cN$ with $E_{VV}=0$ and $E\bar S=0$, and
$F\in\cR$ with $F_{UU}=0$ and $\bar XF=0$. The identities
\eqref{eq:kernel-tests} at $\bar S$ and $\bar X$ give $E\in\cM$, $F\in \cL$,
hence $H=E+F\in\cD$.

\emph{Step 4: normality contradicts the secant signs.} Choose any
decomposition $H_k=Y_k-W_{U,k}+W_{V,k}$ as in \eqref{eq:shadow-normal};
all three sequences are bounded by \cref{lem:normal-bound}. Since $H_k\to H\in\cD$
and $Y_k\perp\cD$, $\ip{H_k}{Y_k}=\ip{H_k-H}{Y_k}\to0$. The kernel
supports give $\bar X_kW_{U,k}=0$ and $\bar S_kW_{V,k}=0$. By
\eqref{eq:secants} and positivity,
\[
 \ip{E_k}{W_{U,k}}=t_k^{-1}\ip{X_k}{W_{U,k}}\ge0,\qquad
 \ip{F_k}{W_{V,k}}=-\sigma t_k^{-1}\ip{S_k}{W_{V,k}}\le0 .
\]
The cross pairings vanish by \eqref{eq:wrong-blocks}:
$|\ip{F_k}{W_{U,k}}|=|\ip{(F_k)_{UU}}{(W_{U,k})_{UU}}|\to0$ and
$|\ip{E_k}{W_{V,k}}|=|\ip{(E_k)_{VV}}{(W_{V,k})_{VV}}|\to0$. Therefore
\[
 1=\normF{H_k}^2=\ip{H_k}{Y_k}-\ip{H_k}{W_{U,k}}+\ip{H_k}{W_{V,k}}
 \le\ip{H_k}{Y_k}-\ip{F_k}{W_{U,k}}+\ip{E_k}{W_{V,k}}\longrightarrow0,
\]
a contradiction.
\end{proof}

\subsection{Normal modulus and exact linearization:
(iii)\texorpdfstring{$\Rightarrow$}{=>}(iv)\texorpdfstring{$\Rightarrow$}{=>}(i)
and (iii)\texorpdfstring{$\Rightarrow$}{=>}(v)\texorpdfstring{$\Rightarrow$}{=>}(ii)}
\label{app:rigidity-linearization}

\begin{proof}[Proof of \textup{(iii)}$\Rightarrow$\textup{(iv)}$\Rightarrow$\textup{(i)}]
For $\bar Z\in\Zopt$ and a unit $H\in N_{\Zopt}(\bar Z)$, the projection
characterization of closed convex sets gives
$\dist_F(\bar Z+tH,\Zopt)=t$ for $t\ge0$. For $0<t\le t_\sigma$,
\eqref{eq:canonical-eb} and $\delta_\sigma(\bar Z)=0$ give
$\normF{\delta_\sigma(\bar Z+tH)}/t\ge c_\sigma$, and directional
differentiability lets $t\downarrow0$: $\normF{\delta_\sigma'(\bar Z;H)}
\ge c_\sigma$, so $\mu_{\mathrm{nor}}\ge c_\sigma>0$. On $\ri\Zopt$ the
normal cone is $\cD^\perp$ (\cref{prop:cones}), so
$\mu_{\mathrm{ri}}\ge\mu_{\mathrm{nor}}>0$.
\end{proof}

\begin{proof}[Proof of \textup{(iii)}$\Rightarrow$\textup{(v)}]
We prove $\cC(\bar Z)=\Tan_{\Zopt}(\bar Z)$ at \emph{every} $\bar Z\in\Zopt$.
The inclusion $\Tan_{\Zopt}(\bar Z)\subseteq\cC(\bar Z)$ needs no
hypothesis. Let $t_k\downarrow0$ and $H_k\to H$ with
$\bar Z+t_kH_k\in\Zopt$. Since $\delta_\sigma$ is Lipschitz with some
constant $L_\sigma$ and vanishes on $\Zopt$,
$\normF{\delta_\sigma(\bar Z+t_kH)}\le L_\sigma t_k\normF{H-H_k}=o(t_k)$.
The directional derivative exists, so it must be zero. Conversely, if
$\delta_\sigma'(\bar Z;H)=0$ then $\normF{\delta_\sigma(\bar Z+tH)}=o(t)$.
For small $t$, $\dist_F(\bar Z+tH,\Zopt)\le t\normF H\le t_\sigma$, so
\eqref{eq:canonical-eb} applies and gives
$\dist_F(\bar Z+tH,\Zopt)\le c_\sigma^{-1}\normF{\delta_\sigma(\bar Z+tH)}=o(t)$.
With $Z_t$ a nearest point, $(Z_t-\bar Z)/t\to H$, so
$H\in\Tan_{\Zopt}(\bar Z)$.
\end{proof}

\begin{proof}[Proof of \textup{(v)}$\Rightarrow$\textup{(ii)}]
By \cref{lem:automatic-interior} the identities hold at interior factors,
so fix $\bar S\in\rb\Sopt$ and $E\in\cN$ with $E_{VV}=0$ and $E\bar S=0$.
The primal-anchored endpoint $\bar Z:=\Xsc-\sigma\bar S$ lies in
$\rb\Zopt$ (\cref{lem:product-geometry}). Apply the coarse form
\eqref{eq:coarse-derivative} in case (b). The $UU$ block of
$\Pi_+'(\bar Z;E)$ is $E_{UU}$. The $VV$ block is the positive part of the
$K_{\bar S}K_{\bar S}$ block of $E_{VV}=0$, hence zero. The cross block $Y$
solves $\Xsc|_U\,Y+\sigma Y\bar S_{VV}=\Xsc|_U\,E_{UV}$; since
$E_{UV}\bar S_{VV}=0$, the unique solution is $Y=E_{UV}$. Thus
$\Pi_+'(\bar Z;E)=E$, and $\delta_\sigma'(\bar Z;E)=-\Prow E=0$ because
$E\in\cN$. Hence $E\in\cC(\bar Z)=\Tan_{\Zopt}(\bar Z)\subseteq\cD$, and
$E_{VV}=0$ gives $E\in\cM$.

Dually, fix $\bar X\in\rb\Xopt$ and $F\in\cR$ with $F_{UU}=0$ and
$\bar XF=0$, and consider $\bar Z:=\bar X-\sigma\Ssc\in\rb\Zopt$. In case
(c) of \eqref{eq:coarse-derivative}, the $UU$ block of $\Pi_+'(\bar Z;F)$
is the positive part of the $K_{\bar X}K_{\bar X}$ block of $F_{UU}=0$,
hence zero, and the $VV$ block is zero. The cross block solves
$\bar X_{UU}Y+\sigma Y\,\Ssc|_V=\bar X_{UU}F_{UV}=0$, so $Y=0$ by
\cref{lem:sylvester}(b). Thus $\Pi_+'(\bar Z;F)=0$ and
$\delta_\sigma'(\bar Z;F)=-\Pnull F=0$, so
$F\in\cC(\bar Z)=\Tan_{\Zopt}(\bar Z)\subseteq\cD$, and $F_{UU}=0$ gives
$F\in \cL$.
\end{proof}

\subsection{Examples}
\label{app:not-abstract}

\begin{example}[A firmly nonexpansive map with a flat boundary
normal]\label{ex:not-abstract}
Let $\psi(s)=s^2/(1+s)$ for $s\ge0$ and
\[
 q(t)=\begin{cases}t+\psi(-t),&t<0,\\ t,&0\le t\le1,\\
 t-\psi(t-1),&t>1,\end{cases}
\]
a nondecreasing $1$-Lipschitz function, so $G(t,s):=(q(t),0)$ is firmly
nonexpansive on $\R^2$ with $\Fix G=[0,1]\times\{0\}$. Along
relative-interior normals $(0,h)$ the residual has derivative $(0,-h)$, so
the relative-interior modulus equals one. Yet at the boundary point $(1,0)$
with outward normal $(1,0)$,
$(G-\Id)(1+t,0)=\bigl(-t^2/(1+t),\,0\bigr)$, so the residual vanishes to
first order. Rigidity fails for this map: the implication
(i)$\Rightarrow$(iv) is a statement about PSD geometry, not about averaged
maps.
\end{example}

\begin{example}[Two disks: both identities are needed]\label{ex:two-disks}
Let $p=4$, $U=\spanof\{e_1,e_2\}$, $V=\spanof\{e_3,e_4\}$, and
\[
 \min_{X\succeq0}\ \tfrac12(X_{33}+X_{44})\quad\text{s.t.}\quad
 \tr X=1,\ X_{33}-X_{44}=0,\ X_{34}=0,\ X_{13}-X_{24}=0,\ X_{14}+X_{23}=0 .
\]
The constraints are independent, and $I_4/4$ is strictly feasible. The
optimal value is $0$. On the primal side, the optimal value forces
$X_{33}=X_{44}=0$, so rows $3,4$ of a PSD optimal $X$ vanish and
$\Xopt=\{\bigl[\begin{smallmatrix}A&0\\0&0\end{smallmatrix}\bigr]:A\succeq0,\ \tr A=1\}$.
On the dual side, the trace multiplier equals the dual objective, which is
zero at optimality. Hence the $UU$ slack block vanishes, and PSD-ness kills
the cross block. The two constraints on the $VV$ block allow every
traceless change to $\tfrac12I_2$, so
$\Sopt=\{\bigl[\begin{smallmatrix}0&0\\0&B\end{smallmatrix}\bigr]:B\succeq0,\ \tr B=1\}$.
Both optimal sets are two-dimensional disks with rank-one boundaries. The
anchors are $\Xsc=\tfrac12P_U$ and $\Ssc=\tfrac12P_V$, and $\cM$, $\cL$
are the traceless parts of $\Sym(U)$ and $\Sym(V)$. The cross block of a matrix in
$\cN$ has the form $\bigl[\begin{smallmatrix}a&b\\-b&a\end{smallmatrix}\bigr]$
and that of a matrix in $\cR$ the form
$\bigl[\begin{smallmatrix}a&b\\b&-a\end{smallmatrix}\bigr]$; both are
invertible unless zero. If $E\in\cN$ has $E_{VV}=0$ and $E\bar S=0$, then
$E_{UV}\bar S_{VV}=0$ with $\bar S_{VV}\ne0$ (trace one), so $E_{UV}=0$ and
$E\in\cM$; symmetrically $\bar X_{UU}F_{UV}=0$ forces $F\in \cL$. Thus
\eqref{eq:kernel-tests} holds at every endpoint, including those at which
\emph{both} factors have rank one, and \cref{thm:rigidity} applies for
every $\sigma>0$ although neither optimum is unique.
\end{example}

\begin{remark}[What rigidity does and does not assert]\label{rem:rigidity-scope}
The proof gives positive constants $\mu_{\mathrm{nor}},c_\sigma,t_\sigma$
for the fixed instance, but does not identify them with $\mu_{\mathrm{ri}}$
or furnish a uniform rate over objectives. The proof of
(ii)$\Rightarrow$(iii) is not a uniform Taylor expansion across singular
endpoints. It exploits exact PSD complementarity of the projection
secants. The uniform spectral estimates use the fixed anchor gaps, and no
estimate divides by eigenvalues of the moving optima; the remaining
constants depend on the fixed problem data and the compact optimal sets. Likewise, by \cref{lem:automatic-interior}
every relative-interior derivative $\delta_\sigma'(\bar Z;\cdot)$ has
kernel exactly $\cD$, that is, it is injective on $\cD^\perp$. This
pointwise statement is weaker than the uniform bound
$\mu_{\mathrm{ri}}>0$: the tied-objective instance of
\cref{app:tied-objectives} satisfies \cref{ass:two-sided} and yet has
$\mu_{\mathrm{ri}}=0$.
\end{remark}

\subsection{Relation to second-order limit dynamics}
\label{app:limit-dynamics}

The stall cone $\cC(\bar Z)$ of \eqref{eq:stall-cone} is the central
object of the companion analysis~\cite{kang26arxiv-admmsdp-limitdyn},
which studies ADMM near an arbitrary KKT point under the same
strict-complementarity assumption. There, $\cC(\bar Z)$ is the fixed-point
set of the local first-order dynamics. Every direction
$H\in\cC(\bar Z)\setminus\Tan_{\Zopt}(\bar Z)$ drives a second-order limit
drift $\phi(\bar Z;H)\ne0$ (its Proposition~6). Along such a direction the
residual decays like $t^2$ while the distance to the optimal set stays of
order $t$; these are the slow regions seen in practice. That paper proves that a local linear
growth condition at $\bar Z$ forces $\cC(\bar Z)=\Tan_{\Zopt}(\bar Z)$
(its Corollary~1(3)) and conjectures that the equality is necessary for
local linear convergence. \Cref{thm:rigidity} supplies the converse in the
setwise form: the equality at every endpoint of a compact fixed-point set
is \emph{equivalent} to the uniform error bound (iii). In words, benign
degeneracy means that no endpoint carries a direction of
$\cC(\bar Z)\setminus\Tan_{\Zopt}(\bar Z)$, along which the residual is of
second order while the distance is of first order. \Cref{sec:boolean} shows that exact
Boolean relaxations have none, and \cref{rem:tied-objectives} exhibits one
for nonconstant full-order Boolean objectives with tied minimizers.

%% file: sections/app_modulus.tex
\section{The Boolean Modulus}
\label{app:modulus}

\subsection{Proof of the subspace identities}
\label{app:identities-proof}

\begin{proof}[Proof of \cref{lem:boolean-geometry}]
In coordinates whose first vector is $q$, matrices of $\Sym(V)$ have zero
first row and column, so $\Sym(V)^\perp$ consists of the matrices with zero
$VV$ block, which are $\{qv^\top+vq^\top\}$; this is
\eqref{eq:Fq-perp}. If such a matrix lies in $\cT_d$, its diagonal
vanishes; the diagonal entries are $2q_Av_A$ and every
$q_A=\pm N^{-1/2}\ne0$, so $v=0$, proving
\eqref{eq:transverse-intersection}. For \eqref{eq:boolean-L}, the Boolean
moment identity is
\begin{equation}\label{eq:boolean-lift}
 uu^\top=\sum_{\Gamma\in\cI_{2d}}x_\star^\Gamma A_\Gamma
        =I_N+\sum_{\Gamma\ne\emptyset}x_\star^\Gamma A_\Gamma ,
\end{equation}
since $(uu^\top)_{A,B}=x_\star^Ax_\star^B=x_\star^{A\triangle B}$. If
$H\in\Sym(V)\cap\cR_d$, then $Hu=0$ and $H\perp\cT_d$, so
$0=\ip{uu^\top}H=\tr H$ by \eqref{eq:boolean-lift}; hence $H\perp I_N$ as
well and $H\in\cH_d^\perp$. The reverse inclusion is immediate because
$\cT_d\subseteq\cH_d$. By \eqref{eq:directions}, $\Sym(V)\cap\cR_d=\cL$;
that $\cL$ also equals $\spanof(\Sopt-\Sopt)$ is \cref{lem:product-geometry}
under the instantiation of \cref{sec:boolean}. Taking orthogonal
complements gives
$\cL^\perp=\Sym(V)^\perp+\cH_d=\Sym(V)^\perp+\cT_d$ (subspace sums, not
orthogonal sums), the last equality because
$I_N=uu^\top-\sum_{\Gamma\neq\emptyset}x_\star^\Gamma A_\Gamma$ and
$uu^\top=N qq^\top\in\Sym(V)^\perp$.
\end{proof}

This resolves a distinction worth recording: the primal equality row space
is $\cT_d^\perp$, while the kernel of the Gram-to-polynomial map is
$\cH_d^\perp$. They differ, but their intersections with $\Sym(q^\perp)$
agree.

\subsection{Derivative identifications for the extension}
\label{app:deriv-id}

Fix $\sigma>0$ and $S\in\Sopt$, and let $P_K:=Q_KQ_K^\top$ be the
orthogonal projector onto $K:=\ker(S)\cap V$, with $Q_K$ from
\eqref{eq:eigen-chart}. Recall $\Jext_S$ and $\Psi_S$ from
\eqref{eq:Jextension}. We prove the two identifications used in
\cref{sec:boolean}:
\begin{equation}\label{eq:interior-derivative}
 S\in\ri\Sopt:\qquad
 \Pi_+'(X_\star-\sigma S;\cdot)=\Jext_S,
 \qquad
 \delta_\sigma'(X_\star-\sigma S;\cdot)=-\Psi_S ,
\end{equation}
\begin{equation}\label{eq:true-boundary-derivative}
 S\in\Sopt\ \text{arbitrary},\ K=\ker(S)\cap V:\quad
 \Pi_+'(X_\star-\sigma S;H)
 =\Jext_S(H)+Q_K\,\Pi_+(Q_K^\top HQ_K)\,Q_K^\top .
\end{equation}

\begin{proof}
Apply \eqref{eq:positive-deriv} at $\bar Z=X_\star-\sigma S$ with the
Boolean instantiation of \cref{sec:boolean}: $U=\spanof\{q\}$ is
one-dimensional, so $\Lambda_+=(N)$ in \eqref{eq:eigen-chart} and
$\Omega_{1j}=N/(N+\sigma s_j)$. The $UU$ block
contributes $h_{00}qq^\top$ with $h_{00}=q^\top Hq$. The cross blocks
from $U$ to $K$ and to $\ran S$ assemble into $qw^\top+wq^\top$ with
\[
 w=\sum_{j}\frac{N}{N+\sigma s_j}\,\theta_j\theta_j^\top Hq
  \;+\;P_KHq
  \;=\;N(NI_N+\sigma S)^{-1}P_VHq ,
\]
where the $\theta_j$ are the columns of $Q_-$ in \eqref{eq:eigen-chart}.
The identity holds since $(NI_N+\sigma S)^{-1}$ acts on $K$ as $N^{-1}P_K$
and on the eigenvector $\theta_j$ of $S$ as $(N+\sigma s_j)^{-1}$. Note
$w\perp q$. The
$VV$ block of \eqref{eq:positive-deriv} is
$Q_K\Pi_+(Q_K^\top HQ_K)Q_K^\top$, which vanishes when $K=\{0\}$, i.e.\ when
$S\in\ri\Sopt$ by \cref{lem:product-geometry}. This proves
\eqref{eq:true-boundary-derivative}, and \eqref{eq:interior-derivative}
follows because at interior certificates $X_\star-\sigma S$ is nonsingular,
$\Pi_+$ is differentiable there, and
$\delta_\sigma'=-\Prow\Pi_+'-\Pnull\Pi_-'
=-\bigl(\Prow\Jext_S+\Pnull(\Id-\Jext_S)\bigr)=-\Psi_S$ by
\eqref{eq:res-deriv}.
\end{proof}

\subsection{Proof of the uniform modulus}
\label{app:modulus-proof}

\begin{proof}[Proof of \cref{prop:boolean-modulus}]
Suppose $\Psi_S(H)=0$ and put $E:=\Jext_S(H)$, $F:=H-E$. The two residual
components vanish separately, so $E\in\cT_d$ and $F\in\cR_d$. But
$E\in\Sym(V)^\perp$ by construction, so the constant-diagonal identity
\eqref{eq:transverse-intersection} forces
\[
 E\ \in\ \cT_d\cap\Sym(V)^\perp=\{0\}.
\]
Multiplying $E=0$ by $q$ gives $h_{00}=0$ and $w=0$; the inverse in
\eqref{eq:Jextension} is nonsingular, so $P_VHq=0$, and with $q^\top Hq=0$
this yields $Hq=0$, i.e.\ $H\in\Sym(V)$. Then $F=H\in\cR_d$, so
$H\in\Sym(V)\cap\cR_d=\cL$. Conversely $H\in \cL$ gives $\Jext_S(H)=0$ and
$\Pnull H=0$, so $\Psi_S(H)=0$. This proves $\ker(\Psi_S)=\cL$; in particular
$\sigma_{\min}(\Psi_S|_{\cL^\perp})>0$ for each $S$ by finite dimensionality.
Continuity of $S\mapsto\Psi_S$ (\cref{app:continuity} below) makes this
least singular value continuous, and compactness of $\Sopt$ gives
$\mu_\star>0$. Finally, on $\ri\Sopt$ the extension coincides with the true
derivative, which gives \eqref{eq:boolean-muri}.
\end{proof}

\subsection{Continuity and the compactness step}
\label{app:continuity}

The inverse in \eqref{eq:Jextension} satisfies
$\normop{(NI_N+\sigma S)^{-1}}\le1/N$ uniformly on $\Sopt$, so the
identity $B^{-1}-A^{-1}=B^{-1}(A-B)A^{-1}$ makes
$S\mapsto\Jext_S$, hence $S\mapsto\Psi_S$, Lipschitz on $\Sopt$ in operator
norm. The Lipschitz bound
$|\sigma_{\min}(B_1)-\sigma_{\min}(B_2)|\le\normop{B_1-B_2}$
(\cref{app:spectral}) then makes
$S\mapsto\sigma_{\min}(\Psi_S|_{\cL^\perp})$ continuous, which is the
compactness step used in the proof of \cref{prop:boolean-modulus}. The
compactness argument applies to the continuous family
$\{\Psi_S\}_{S\in\Sopt}$ on the closed set $\Sopt$;
\eqref{eq:interior-derivative} identifies this family with the true
derivative on $\ri\Sopt$, which is all that \eqref{eq:boolean-muri}
requires. In the language of \cref{app:deriv-validity}, $\Psi_S$ is the
primal-anchored continuous extension of the interior derivative; its dual
twin, anchored at $\Ssc$ with a varying primal optimum, is trivial here
because $\Xopt$ is a point.

\subsection{Tied minimizers: the error bound fails at full order}
\label{app:tied-objectives}

This section proves the claims of \cref{rem:tied-objectives}. Work at
full order $d=n$, where $N=2^n$ and the normalized evaluation vectors
$b_n(a)/\sqrt N$, $a\in\{\pm1\}^n$, form an orthonormal basis
(\cref{app:full-order}); let $Q$ be the orthogonal matrix with these
columns and transform all data by $Q$. Every XOR-class matrix $A_\Gamma$
is diagonal in this basis. Hence the feasible set becomes the diagonal
matrices $\diag(z)$ with $\sum_az_a=N$, and the moment matrix of a
distribution $\pi$ on the cube is $\diag(N\pi)$. The objective becomes
$\sum_a\pi_af(a)$, and the data of \cref{sec:rigidity} are
\[
\begin{gathered}
 \cN=\{\text{traceless diagonal matrices}\},\qquad
 \cR=\{\text{off-diagonal symmetric matrices}\}\oplus\spanof\{I_N\},\\
 X_b=I_N,\qquad \widehat C=\diag\bigl(f(a)/N\bigr).
\end{gathered}
\]
Let $f$ be nonconstant with minimizers $a\ne b$ and a nonminimizer $c$,
and write $f_\star:=\min f$. Then $\Xopt$ contains every $\diag(N\pi)$
with $\pi$ supported on the minimizers, so it is not a singleton. The
matrix $\bar S:=\diag\bigl((f(a')-f_\star)/N\bigr)_{a'}$ is an optimal
slack: $\bar S-\widehat C=-(f_\star/N)I_N\in\cR$, $\bar S\succeq0$, and
$\bar S$ vanishes on the minimizers. Finally, $\bigl(\Xsc,\Ssc\bigr):=
\bigl(\diag(N\pi_{\mathrm u}),\bar S\bigr)$ with $\pi_{\mathrm u}$ uniform
on the minimizers is a strictly complementary pair. Thus
\cref{ass:two-sided} holds, with $U$ spanned by the minimizer coordinates
and $V$ by the others.

\paragraph{The instance $f=1-x_1$, $n=2$.} With the minimizers ordered
first, $\Xopt=\{\diag(\alpha,4-\alpha,0,0):0\le\alpha\le4\}$,
$\Sopt=\{\diag(0,0,\tfrac12,\tfrac12)+s(e_3e_4^\top+e_4e_3^\top):|s|\le\tfrac12\}$
(the only elements of $\cR$ supported on $V$ are the off-diagonal $VV$
matrices), $\cM=\spanof\{\diag(1,-1,0,0)\}$, $\cL=\spanof\{e_3e_4^\top+e_4e_3^\top\}$.

\paragraph{The endpoint and the direction.} Take the primal optimum
$\bar X:=Ne_ae_a^\top$ (the moment matrix of the point mass at $a$), the
slack $\bar S$ above, the endpoint $\bar Z:=\bar X-\sigma\bar S$, and the
direction $H:=e_be_c^\top+e_ce_b^\top$. Then $H\in\cR$, and $H\perp\cD$ because
$H$ is a cross term between $U\ni e_b$ and $V\ni e_c$ while $\cM\subseteq
\Sym(U)$ and $\cL\subseteq\Sym(V)$. Hence the nearest point of $\aff\Zopt$ to
$\bar Z+tH$ is $\bar Z$, and $\dist_F(\bar Z+tH,\Zopt)=\normF{tH}=\sqrt2\,t$.
The only spectral block of $\bar Z$ changed by $tH$ is the $\{b,c\}$ block
$\bigl(\begin{smallmatrix}0&t\\t&-\sigma\gamma\end{smallmatrix}\bigr)$,
$\gamma:=(f(c)-f_\star)/N>0$, whose single positive eigenvalue is
\begin{equation}\label{eq:tied-lambda}
 \lambda_+(t)=\frac{\sqrt{\sigma^2\gamma^2+4t^2}-\sigma\gamma}{2}
 =\frac{t^2}{\sigma\gamma}+O(t^4),
\end{equation}
so $B(t):=\Pi_+(\bar Z+tH)-\bar X$ is rank one with $\normF{B(t)}=\lambda_+(t)$.

\paragraph{The residual.} Since $\Prow\bar X=(\tr\bar X/N)I_N=I_N=X_b$,
$\Pnull(\sigma\widehat C-\sigma\bar S)=\Pnull(\sigma f_\star I_N/N)=0$, and
$\Pnull H=0$,
\[
 \delta_\sigma(\bar Z+tH)
 =-\Prow B(t)-\Pnull\bigl(tH-B(t)\bigr)
 =(\Id-2\Prow)B(t),
\]
and $\Id-2\Prow$ is an orthogonal reflection, so
$\normF{\delta_\sigma(\bar Z+tH)}=\lambda_+(t)$ exactly. Therefore
\begin{equation}\label{eq:tied-ratio}
 \frac{\normF{\delta_\sigma(\bar Z+tH)}}{\dist_F(\bar Z+tH,\Zopt)}
 =\frac{\lambda_+(t)}{\sqrt2\,t}\longrightarrow0,
\end{equation}
and no local error bound \eqref{eq:canonical-eb} holds at $\bar Z$. For
$f=1-x_1$ one has $\gamma=\tfrac12$, which is the instance quoted in
\cref{rem:tied-objectives}.

\paragraph{Which condition fails.} The direction $H$ lies in $\cR$, has
$H_{UU}=0$, and satisfies $\bar XH=0$ (its rows avoid $a$), but
$H\notin \cL$. So the dual identity (the second identity in
\eqref{eq:kernel-tests}, which tests dual directions at a primal boundary
point) fails at $\bar X$, which is a relative-boundary point of $\Xopt$.
Consistently, let $X_\eps\in\ri\Xopt$ be the diagonal optimum with entry
$\eps$ at every minimizer other than $a$ (for $f=1-x_1$,
$X_\eps=\diag(N-\eps,\eps,0,0)$). Along the interior endpoints
$X_\eps-\sigma\bar S$, the unit direction $H/\sqrt2\in\cD^\perp$ has
$\normF{\delta_\sigma'(X_\eps-\sigma\bar S;H/\sqrt2)}=\eps/(\eps+\sigma\gamma)\to0$.
Hence $\mu_{\mathrm{ri}}=0$, although each interior derivative has kernel
exactly $\cD$. Moreover $\Pi_+'(\bar Z;H)=0$, so $\delta_\sigma'(\bar Z;H)=
-\Pnull H=0$ and $H\in\cC(\bar Z)\setminus\Tan_{\Zopt}(\bar Z)$. This is
a second-order-dominant direction in the sense of
\cite{kang26arxiv-admmsdp-limitdyn}: along it the residual is quadratic
in the distance to the optimal set.

%% file: sections/app_algorithms.tex
\section{Algorithmic Transfer}
\label{app:algorithms}

\subsection{Error bound plus averagedness: the rate lemma}
\label{app:rate-lemma}

\begin{lemma}[Setwise error bound plus averagedness]\label{lem:averaged-eb}
Let $T$ be $\alpha$-averaged, $0<\alpha<1$, in a norm $\|\cdot\|_G$ with
$g_{\min}\|\cdot\|^2\le\|\cdot\|_G^2\le g_{\max}\|\cdot\|^2$,
$0<g_{\min}\le g_{\max}$, on a finite-dimensional space, suppose $\Fix T\ne\emptyset$, and suppose
$\|Tz-z\|\ge c\,\dist(z,\Fix T)$ whenever $\dist(z,\Fix T)\le t_0$. Then
every orbit converges to a point of $\Fix T$; it eventually enters the
error-bound neighborhood, after which $\dist_G(z^k,\Fix T)$ contracts
Q-linearly and the iterates converge R-linearly.
\end{lemma}

We prove the following quantitative form. Under the hypotheses of
\cref{lem:averaged-eb}, an orbit starting with
$\dist(z^0,\Fix T)\le t_0\sqrt{g_{\min}/g_{\max}}$ satisfies
\begin{equation}\label{eq:rate-metric}
 \dist_G(z^k,\Fix T)\le\rho^k\dist_G(z^0,\Fix T),\qquad
 \rho^2=1-\frac{1-\alpha}{\alpha}\,\bar c^{\,2}\frac{g_{\min}}{g_{\max}}
 \ \in(0,1),
\end{equation}
for any $0<\bar c\le c$ small enough that $\rho>0$. In the Euclidean norm,
$\dist(z^k,\Fix T)\le\sqrt{g_{\max}/g_{\min}}\,\rho^k\dist(z^0,\Fix T)$,
and the iterates converge R-linearly to a point of $\Fix T$. Every orbit
converges to a fixed point and eventually enters such a neighborhood, so
the conclusions hold eventually from any initialization.

\begin{proof}
The averagedness
inequality~\cite{bauschke17springer-convex-analysis-hilbert-spaces} is
\[
 \|Tz-z_\star\|_G^2\le\|z-z_\star\|_G^2
 -\frac{1-\alpha}{\alpha}\|Tz-z\|_G^2\qquad(z_\star\in\Fix T).
\]
Taking a $G$-nearest fixed point and using norm equivalence with the error
bound gives
$\dist_G(Tz,\Fix T)^2\le\rho^2\dist_G(z,\Fix T)^2$ whenever
$\dist(z,\Fix T)\le t_0$. The initial condition gives
$\dist_G(z^0,\Fix T)\le\sqrt{g_{\min}}\,t_0$, and Fej\'er monotonicity
preserves it, so $\dist(z^k,\Fix T)\le t_0$ for every $k$; this proves
\eqref{eq:rate-metric} and its Euclidean form. Nonexpansiveness gives
$\|Tz-z\|_G\le2\dist_G(z,\Fix T)$, so the step lengths have a geometrically
summable tail and the iterates converge R-linearly to some point of
$\Fix T$. For an arbitrary starting point, averagedness gives boundedness
and square summability of $\|Tz^k-z^k\|_G$. A cluster point is fixed by
continuity. Fej\'er monotonicity relative to it gives convergence of the
whole sequence, which therefore eventually enters the neighborhood above.
\end{proof}

Q-linear contraction is asserted in the $G$ metric; norm equivalence yields
R-linear distance and iterate bounds in the original coordinates, not
necessarily one-step contraction in every equivalent norm.

\subsection{ADMM and exact sGS--ADMM share one shadow map}
\label{app:shadow-map}

The shared shadow representation of this subsection was established in the
afterimage paper~\cite[Table~1 and Appendix~A.2]{kang26arxiv-afterimage};
we spell it out in the present notation, since the error-bound transfer of
\cref{sec:algorithms} acts on it. Note that for $\tau\ne1$ the map
$T_{\sigma,\tau}$ is \emph{not} the scalar relaxation
$(1-\tau)\Id+\tau T_{\sigma,1}$; the multiplier step rescales the two
orthogonal residual components blockwise. For $\sigma>0$ and $0<\tau<2$, exact
ADMM on the dual SDP
\eqref{eq:sdp-pair} is
\begin{align}
 \lambda^{k+1}&=(\cA\cA^*)^{-1}
    \bigl[\sigma^{-1}b-\cA(\sigma^{-1}X^k+S^k-C)\bigr],\label{eq:raw-admm-y}\\
 S^{k+1}&=\Pi_+\bigl(C-\cA^*\lambda^{k+1}-\sigma^{-1}X^k\bigr),\label{eq:raw-admm-s}\\
 X^{k+1}&=X^k+\tau\sigma(\cA^*\lambda^{k+1}+S^{k+1}-C).\label{eq:raw-admm-x}
\end{align}
Define the shadow $Z^k:=\Pnull X^k+X_b-\sigma\Prow S^k-\sigma\Pnull C$.
Substitution in the first two equations gives
\begin{equation}\label{eq:raw-shared-S}
 S^{k+1}=-\sigma^{-1}\Pi_-(Z^k),
\end{equation}
because $\sigma(C-\cA^*\lambda^{k+1})-X^k=-Z^k$ (check the $\cR$ and $\cN$
components separately using \eqref{eq:raw-admm-y}); and the multiplier
update reduces to the filter
\begin{equation}\label{eq:admm-filter}
 X^{k+1}=(1-\tau)X^k+\tau\,\Pi_+(Z^k).
\end{equation}
Projecting the next shadow onto $\cR$ and $\cN$ then yields
$\Prow Z^{k+1}=\Prow Z^k-\rpar(Z^k)$ and
$\Pnull Z^{k+1}=\Pnull Z^k-\tau\,\rper_\sigma(Z^k)$, which is the shadow
recursion $T_{\sigma,\tau}$ of \cref{sec:algorithms}.

The exact sGS variant~\cite{chen17mp-sgsadmm} inserts a second equality
solve: $\lambda^{k+1/2}$ by \eqref{eq:raw-admm-y}, then $S^{k+1}$ by
\eqref{eq:raw-admm-s} with $\lambda^{k+1/2}$, then $\lambda^{k+1}$ by
\eqref{eq:raw-admm-y} with $S^{k+1}$, then \eqref{eq:raw-admm-x}. The same
shadow gives \eqref{eq:raw-shared-S}, while the last two equations give
\begin{equation}\label{eq:sgs-X}
 X^{k+1}=(\Id-\tau\Prow)X^k+\tau X_b+\tau\sigma\Pnull(S^{k+1}-C),
\end{equation}
whose $\cN$ projection again yields
$\Pnull Z^{k+1}=\Pnull Z^k-\tau\,\rper_\sigma(Z^k)$; the $\cR$ projection of
the shadow depends only on $S^{k+1}$ and matches ADMM. Thus both algorithms
share the recursion $T_{\sigma,\tau}$; the shadow records only the pair
$(\Pnull X^k,\Prow S^k)$, not all of $(X^k,S^k)$.

\subsection{Parameter-dependent metric and rates for the shadow map}
\label{app:shadow-rates}

\begin{proposition}[A parameter-dependent metric for the shadow map]\label{prop:admm-metric}
For $0<\tau<2$, $T_{\sigma,\tau}$ is $\alpha_\tau$-averaged with
$\alpha_\tau=\max\{1,\tau\}/2$ in the metric
$\mathsf G_\tau:=\Prow+\tau^{-1}\Pnull$, whose eigenvalues bound it between
$\min\{1,\tau^{-1}\}$ and $\max\{1,\tau^{-1}\}$ times the Frobenius metric.
\end{proposition}

\begin{proof}[Proof of \cref{prop:admm-metric}]
Write $T_{\sigma,\tau}=\Id-\mathsf D_\tau R_0$ with $R_0:=\Id-T_{\sigma,1}$,
which is firmly nonexpansive because the canonical DR map
is~\cite{bauschke17springer-convex-analysis-hilbert-spaces}. Put $e=Z-W$ and
$v=R_0Z-R_0W$, so $\ip ev\ge\normF v^2$. With
$\mathsf G_\tau=\mathsf D_\tau^{-1}$, self-adjointness gives
\[
 \|e-\mathsf D_\tau v\|_{\mathsf G_\tau}^2
 =\|e\|_{\mathsf G_\tau}^2-2\ip ev+\ip v{\mathsf D_\tau v}
 \le\|e\|_{\mathsf G_\tau}^2
 -\Bigl(\frac{2}{\max\{1,\tau\}}-1\Bigr)\|\mathsf D_\tau v\|_{\mathsf G_\tau}^2,
\]
using $\ip v{\mathsf D_\tau v}\le\max\{1,\tau\}\normF v^2\le
\max\{1,\tau\}\ip ev$. This is the averagedness inequality with
$\alpha_\tau=\max\{1,\tau\}/2<1$ for $\tau<2$; the metric bounds are the two
eigenvalues of $\mathsf G_\tau$.
\end{proof}

\begin{proof}[Proof of \cref{cor:admm}]
Combine \eqref{eq:admm-eb}, \cref{prop:admm-metric}, and
\cref{lem:averaged-eb}: the shadow distance contracts eventually Q-linearly
in $\mathsf G_\tau$ and the shadow converges R-linearly to some
$\bar Z=\bar X-\sigma\bar S\in\Zopt$ (in the Boolean case $\bar X=X_\star$).
For the raw variables: projections are Lipschitz, so $\Pi_+(Z^k)\to\bar X$
and $S^{k+1}\to\bar S$ R-linearly by \eqref{eq:raw-shared-S}. For ADMM,
\eqref{eq:admm-filter} is a stable linear recurrence with factor
$|1-\tau|<1$. Hence $X^k\to\bar X$ R-linearly at any rate
$\widehat\rho\in(\max\{\rho,|1-\tau|\},1)$; the enlargement also absorbs
the $k\rho^k$ term when the two rates coincide. Then \eqref{eq:raw-admm-y}
transfers R-linear convergence to $\lambda^{k+1}$, whose limit is the
unique multiplier of $\bar S$. For sGS--ADMM, \eqref{eq:sgs-X} gives the
exact contraction $\Prow X^{k+1}-X_b=(1-\tau)(\Prow X^k-X_b)$. The
component $\Pnull X^k=\Pnull Z^k+\sigma\Pnull C$ is read off the shadow.
Together with the equality-solve formulas, this proves R-linear
convergence of all raw sequences.
\end{proof}

\subsection{PDHG: chart, factorization, and metric}
\label{app:pdhg-proofs}

Recall the chart of \cref{sec:algorithms}:
\begin{equation}\label{eq:pdhg-chart}
 \Phi(X,\lambda)=(Z,r),\qquad
 Z:=X-\eta_x(C-\cA^*\lambda),\qquad r:=\cA X-b .
\end{equation}
The chart is consistent with, but does not require, the known PDHG--DR
equivalences~\cite{oconnor20mp-pdhg-drs}. Throughout, $T_{\mathrm P}$ is
the PDHG map of \cref{sec:algorithms}, and $\Zopt^{\eta_x}$,
$\delta_{\eta_x}$, $c_{\eta_x}$, $t_{\eta_x}$ denote the fixed-point set,
the residual, and the constants of \eqref{eq:canonical-eb} at
$\sigma=\eta_x$.

\begin{proposition}[Exact PDHG transfer]\label{prop:pdhg-transfer}
Under surjectivity of $\cA$ and $\eta_x,\eta_y>0$, the chart
\eqref{eq:pdhg-chart} is an affine bijection mapping
$\Wopt=\Fix T_{\mathrm P}$ onto $\Zopt^{\eta_x}\times\{0\}$, and
\eqref{eq:pdhg-factor} holds with the linear isomorphism
\begin{equation}\label{eq:pdhg-L}
 \mathsf L_{\eta_x,\eta_y}(R^\parallel,R^\perp,r)
 =\begin{pmatrix}
 R^\parallel-R^\perp-\cA^\dagger r-\eta_x\eta_y\cA^*(2\cA R^\parallel-r)\\[1mm]
 \cA R^\parallel-r
 \end{pmatrix},
 \ (R^\parallel,R^\perp,r)\in\cR\times\cN\times\R^m .
\end{equation}
Consequently, if the canonical residual with $\sigma=\eta_x$ satisfies
\eqref{eq:canonical-eb}, then there are $c_{\mathrm P},t_{\mathrm P}>0$ with
$\|\delta_{\mathrm P}(w)\|\ge c_{\mathrm P}\dist(w,\Wopt)$ whenever
$\dist(w,\Wopt)\le t_{\mathrm P}$, where
$\delta_{\mathrm P}:=T_{\mathrm P}-\Id$.
\end{proposition}
\begin{proof}
The linear part of $\Phi$ is
$\mathsf M_\Phi=\bigl[\begin{smallmatrix}\Id&\eta_x\cA^*\\
\cA&0\end{smallmatrix}\bigr]$. If $X+\eta_x\cA^*\lambda=0$ and $\cA X=0$,
then $\cA\cA^*\lambda=0$, so $\lambda=0$ and $X=0$. Hence the square map is
invertible, with explicit inverse
$X=X_b+\cA^\dagger r+\Pnull(Z+\eta_xC)$,
$\lambda=\eta_x^{-1}(\cA\cA^*)^{-1}\cA(Z-X+\eta_xC)$. At a fixed point,
$\cA X=b$ and $X=\Pi_+(X-\eta_xS)$ with $S=C-\cA^*\lambda$---exactly the
KKT conditions---so the image of the fixed set is as claimed. Since
$X^+=\Pi_+(Z)$, decomposing by $\Prow,\Pnull$ gives
\[
 X^+-X=\rpar(Z)-\rper_{\eta_x}(Z)-\cA^\dagger r,\qquad
 r^+=\cA\,\rpar(Z),\qquad
 \lambda^+-\lambda=\eta_y\bigl(r-2\cA\,\rpar(Z)\bigr),
\]
and $Z^+-Z=X^+-X+\eta_x\cA^*(\lambda^+-\lambda)$ proves
\eqref{eq:pdhg-factor}. For injectivity of $\mathsf L_{\eta_x,\eta_y}$,
suppose the image is zero. The second component gives $r=\cA R^\parallel$.
Substituting into the first gives
$-R^\perp-\eta_x\eta_y\cA^*\cA R^\parallel=0$, whose two terms lie in the
orthogonal spaces $\cN$ and $\cR$. Hence $R^\perp=0$ and
$\cA R^\parallel=0$, which forces $R^\parallel=0$ and then $r=0$. Equal
finite dimensions give the isomorphism; neither invertibility uses the
step-size inequality. Finally, let
$\ell:=\sigma_{\min}(\mathsf L_{\eta_x,\eta_y})>0$,
$a_\Phi:=\sigma_{\min}(\mathsf M_\Phi)>0$,
$b_\Phi:=\normop{\mathsf M_\Phi}$. Whenever
$\dist_F(Z,\Zopt^{\eta_x})\le t_{\eta_x}$, orthogonality of the residual
components gives
\[
 \|\Phi T_{\mathrm P}\Phi^{-1}(Z,r)-(Z,r)\|
 \ge\ell\sqrt{\|\delta_{\eta_x}(Z)\|_F^2+\|r\|_2^2}
 \ge\ell\min\{c_{\eta_x},1\}\,
 \dist\bigl((Z,r),\Zopt^{\eta_x}\times\{0\}\bigr).
\]
Moreover $\Phi$ is affine, so
$a_\Phi\dist(w,\Wopt)\le\dist(\Phi w,\Phi\Wopt)\le b_\Phi\dist(w,\Wopt)$,
and the transformed displacement is $\mathsf M_\Phi\delta_{\mathrm P}(w)$.
One may therefore take
$c_{\mathrm P}=(a_\Phi/b_\Phi)\,\ell\min\{c_{\eta_x},1\}$ and
$t_{\mathrm P}=t_{\eta_x}/b_\Phi$.
\end{proof}

For a fixed instance, invertibility already supplies the positive
constants; uniform invertibility would be needed only for constants uniform
over a family of data and step sizes.

\begin{proposition}[PDHG parameter-dependent metric]\label{prop:pdhg-metric}
If $\theta:=\sqrt{\eta_x\eta_y}\,\normop{\cA}<1$, then $T_{\mathrm P}$ is
firmly nonexpansive in
$\mathsf G_{\mathrm P}:=\bigl[\begin{smallmatrix}\eta_x^{-1}\Id&\cA^*\\
\cA&\eta_y^{-1}I_m\end{smallmatrix}\bigr]\succ0$, and one may take
$g^{\mathrm P}_{\min}=(1-\theta)/\max\{\eta_x,\eta_y\}$ and
$g^{\mathrm P}_{\max}=(1+\theta)/\min\{\eta_x,\eta_y\}$ in
\cref{lem:averaged-eb}.
\end{proposition}
\begin{proof}
The cross term obeys
$2|\ip{\cA X}\lambda|\le\theta(\eta_x^{-1}\normF X^2+\eta_y^{-1}
\normtwo\lambda^2)$, which gives the metric bounds. Define the KKT
operator
$\mathsf F_{\mathrm{KKT}}(X,\lambda):=\bigl(C-\cA^*\lambda+N_{\Sym_+^p}(X),\
\cA X-b\bigr)$. It is monotone, because its linear part
$(X,\lambda)\mapsto(-\cA^*\lambda,\cA X)$ is skew-adjoint and the
normal-cone map $N_{\Sym_+^p}$ is monotone. Direct substitution from the
PDHG update gives
$0\in\mathsf F_{\mathrm{KKT}}(w^+)+\mathsf G_{\mathrm P}(w^+-w)$, the
preconditioned proximal-point form of
PDHG~\cite{chambolle11jmiv-pdhg,oconnor20mp-pdhg-drs}; monotonicity applied
to two such inclusions gives
$\|w^+-\bar w^+\|_{\mathsf G_{\mathrm P}}^2\le
\ip{w^+-\bar w^+}{w-\bar w}_{\mathsf G_{\mathrm P}}$, which is firm
nonexpansiveness.
\end{proof}

\begin{proof}[Proof of \cref{cor:pdhg}]
Apply \cref{lem:averaged-eb} with $\alpha=1/2$,
\cref{prop:pdhg-transfer,prop:pdhg-metric}. R-linear convergence of the
slack follows from $S^k=C-\cA^*\lambda^k$ and Lipschitz continuity.
\end{proof}

%% file: sections/app_genericity.tex
\section{Genericity}
\label{app:genericity}

Two clarifications on scope. First, genericity is asserted in the full
coefficient space $\cP_n$; a structured subfamily (sparse, low-degree, or
symmetric) is measure-zero and needs its own relative-genericity argument.
Second, full row rank of $\cA_d$ is a representation choice made once and
for all in \cref{prop:encoding}, not a property supplied by genericity.
Under an arbitrary split in \eqref{eq:wGamma}, the vectors $w_A$ can be
linearly dependent, and the kernel of $S_{\mathrm{int}}$ can be strictly
larger than $\spanof\{u\}$. The convention stated in \cref{sec:genericity}
is what makes them a basis of $u^\perp$.

\subsection{Full order: proof of \texorpdfstring{\cref{prop:eventual}}{the full-order proposition}}
\label{app:full-order}

\begin{proof}[Proof of \cref{prop:eventual}]
Define the point indicators
$\ell_a(x):=2^{-n}\prod_{i=1}^n(1+a_ix_i)
=2^{-n}\sum_{T\subseteq[n]}a^Tx^T$. Modulo $I_{\mathrm B}$,
$\ell_a^2\equiv\ell_a$ and $\sum_a\ell_a\equiv1$. At full order the moment
PSD condition tests the square of every Boolean function, so
$\pi_a:=L_y(\ell_a)=L_y(\ell_a^2)\ge0$ and $\sum_a\pi_a=1$; the
interpolation identity $h\equiv\sum_ah(a)\ell_a$ gives
$L_y(h)=\sum_a\pi_ah(a)$, i.e.\ the simplex representation. The vectors
$b_n(a)$ are mutually orthogonal because
$b_n(a)^\top b_n(b)=\prod_i(1+a_ib_i)=2^n\mathbf 1_{\{a=b\}}$, which proves
uniqueness of the coefficients and, for uniquely minimized $f$, uniqueness
of the optimizer $b_n(x_\star)b_n(x_\star)^\top$.

For upward persistence, let $d_0$ be an order at which $f$ is uniquely
rank-one exact and let $d_0<d\le n$. An order-$d$ optimizer truncates to an
order-$d_0$ feasible matrix with the same objective value, since
$\deg f\le2d_0$. Hierarchy monotonicity and the feasible lift of $x_\star$
give $p_{d_0}(f)\le p_d(f)\le f(x_\star)=p_{d_0}(f)$. So the truncation is
the unique lower-order optimum; in particular $y_{\{i\}}=x_{\star,i}$ for
all $i$. Choose Gram vectors $g_A$ of the higher-order optimizer with
$\ip{g_A}{g_B}=y_{A\triangle B}$ and $\|g_A\|=1$. Whenever
$A,A\triangle\{i\}\in\cI_d$,
$\ip{g_A}{g_{A\triangle\{i\}}}=x_{\star,i}\in\{\pm1\}$, so equality in
Cauchy--Schwarz gives $g_{A\triangle\{i\}}=x_{\star,i}\,g_A$. Adding the
elements of $A$ one at a time gives $g_A=x_\star^A\,g_\emptyset$. Hence the
optimizer is $b_d(x_\star)b_d(x_\star)^\top$.

For the certificate: with $g:=f-f(x_\star)$ and $c_a:=2^{-n}b_n(a)$, set
\begin{equation}\label{eq:full-sc}
 S_{\mathrm{full}}:=\sum_{a\ne x_\star} g(a)\,c_ac_a^\top .
\end{equation}
One has $b_n(x)^\top c_ac_a^\top b_n(x)\equiv\ell_a(x)^2$, so
$S_{\mathrm{full}}$ represents
$\sum_{a\ne x_\star}g(a)\ell_a^2\equiv\sum_ag(a)\ell_a\equiv g$; all
displayed weights $g(a)$ are positive by unique minimality, and
orthogonality of the $b_n(a)$ gives
$\ker(S_{\mathrm{full}})=\spanof\{b_n(x_\star)\}$.
\end{proof}

\subsection{Fixed-order cells: proof of \texorpdfstring{\cref{thm:cell-generic}}{the cell theorem}}
\label{app:cells-proof}

Recall the cells of \cref{sec:genericity}: for $u=b_d(a)$,
\begin{align}
 \cQ(a,d)&:=\{g\in\cP_{\le2d}:g(a)=0\},\notag\\
 \cE(a,d)&:=\Bigl\{g\in\cQ(a,d):
 \argmin_{X\in\cX_d}\ip{C_g}X=\{uu^\top\}\Bigr\},\label{eq:cells}\\
 \cE_{\mathrm{sc}}(a,d)&:=\{g\in\cQ(a,d):\exists S\succeq0,\
 \cG_d(S)=g,\ \ker(S)=\spanof\{u\}\}.\notag
\end{align}
Suppose a certificate $S$ of $\cE_{\mathrm{sc}}(a,d)$ exists. Every
feasible $X$ has value $\ip SX\ge0$, while $uu^\top$ has value zero.
Equality forces $\ran(X)\subseteq\ker(S)=\spanof\{u\}$, and
$X_{\emptyset,\emptyset}=1$ then gives $X=uu^\top$. Hence
$\cE_{\mathrm{sc}}(a,d)\subseteq\cE(a,d)$. In particular, the definition of
$\cE_{\mathrm{sc}}(a,d)$ may ask only that some PSD Gram matrix with kernel
$\spanof\{u\}$ represents $g$; optimality is automatic.

The family \eqref{eq:wGamma} gives a linear right inverse of $\cG_d$ on
$\cQ(a,d)$. Each $w_\Gamma$ satisfies $w_\Gamma^\top u=0$ and
$\cG_d(w_\Gamma w_\Gamma^\top)=\psi_\Gamma:=2(1-a^\Gamma x^\Gamma)$. The
$|\cI_{2d}|-1$ polynomials $\psi_\Gamma$ form a basis of $\cQ(a,d)$. Define
the section
$\mathscr B_{a,d}(\sum_\Gamma t_\Gamma\psi_\Gamma):=\sum_\Gamma t_\Gamma
w_\Gamma w_\Gamma^\top$. It obeys $\cG_d(\mathscr B_{a,d}(h))=h$ and
$\mathscr B_{a,d}(h)u=0$. Since $t_\Gamma=-a^\Gamma h_\Gamma/2$ and
$\normop{w_\Gamma w_\Gamma^\top}=2$,
\begin{equation}\label{eq:section-bound}
 \normop{\mathscr B_{a,d}(h)}\le\sum_{\Gamma\ne\emptyset}|h_\Gamma|
 \le\sqrt{|\cI_{2d}|-1}\ \normc h ,
\end{equation}
where $\normc h$ is the Euclidean norm of the coefficient vector. By the
split convention of \cref{sec:genericity}, $w_A=e_A-a^Ae_\emptyset$, and
these $N-1$ vectors are orthogonal to $u$ and linearly independent. Hence
the interior certificate $S_{\mathrm{int}}$ of \eqref{eq:wGamma} satisfies
$S_{\mathrm{int}}\succeq0$, $\ker(S_{\mathrm{int}})=\spanof\{u\}$, and
$\cG_d(S_{\mathrm{int}})=r_{a,d}:=2\sum_{\emptyset\ne A\in\cI_d}(1-a^Ax^A)$.
In particular $r_{a,d}\in\cE_{\mathrm{sc}}(a,d)$, so the cell is nonempty.

\begin{proof}[Proof of \cref{thm:cell-generic}]
\emph{Openness.} Fix $g\in\cE_{\mathrm{sc}}(a,d)$, and let $S$ be a
certificate for $g$ that is positive definite on $u^\perp$. Write $\kappa>0$
for the smallest eigenvalue of $S$ restricted to $u^\perp$.
For $h\in\cQ(a,d)$, the matrix $S+\mathscr B_{a,d}(h)$ still annihilates
$u$ and represents $g+h$. By \eqref{eq:section-bound} it stays positive
definite on $u^\perp$ whenever $\normc h<\kappa/(2\sqrt{|\cI_{2d}|-1})$.
Hence $g+h\in\cE_{\mathrm{sc}}(a,d)$.

\emph{Density.} If $g\in\cE(a,d)$, primal Slater gives an optimal Gram
matrix $S\succeq0$ representing $g$, and $u^\top Su=g(a)=0$ forces $Su=0$.
For every $\eps>0$, $S+\eps S_{\mathrm{int}}$ represents $g+\eps r_{a,d}$
and is positive definite on $u^\perp$, so
$g+\eps r_{a,d}\in\cE_{\mathrm{sc}}(a,d)$, and $g+\eps r_{a,d}\to g$ as
$\eps\to0$.

\emph{Semialgebraicity.} With $P_{u^\perp}=I_N-uu^\top/N$, membership in
$\cE_{\mathrm{sc}}(a,d)$ is the existential first-order formula
$\exists S,t:\ \cG_d(S)=g,\ Su=0,\ S-tP_{u^\perp}\succeq0,\ t>0$; PSD-ness
is finitely many minor inequalities. Membership in $\cE(a,d)$ is
$g\in\cQ(a,d)$ together with
$\nexists X\in\cX_d:\ \normF{X-uu^\top}^2>0,\ \ip{C_g}X\le0$. The
Tarski--Seidenberg theorem and closure under complements make both cells
semialgebraic~\cite{bochnak98book-rag}.

\emph{Dimension.} Abbreviate $\cE_{\mathrm{sc}}:=\cE_{\mathrm{sc}}(a,d)$ and
$\cE:=\cE(a,d)$; density gives
$\cE_{\mathrm{sc}}\subseteq\cE\subseteq\cl\cE_{\mathrm{sc}}$. The
semialgebraic frontier theorem gives
$\dim(\cl\cE_{\mathrm{sc}}\setminus\cE_{\mathrm{sc}})<\dim\cE_{\mathrm{sc}}$
and $\dim\cl\cE_{\mathrm{sc}}=\dim\cE_{\mathrm{sc}}$~\cite[Propositions~2.8.2
and~2.8.13]{bochnak98book-rag}. Hence
$\dim(\cE\setminus\cE_{\mathrm{sc}})<\dim\cE$. The open neighborhood of
$r_{a,d}$ shows $\dim\cE=\dim\cQ(a,d)=|\cI_{2d}|-1$. The exceptional set
therefore has zero Lebesgue measure in $\cQ(a,d)$.
\end{proof}

\subsection{Ambient genericity}
\label{app:ambient}

\begin{proposition}[Simultaneous SC at every uniquely rank-one exact order]
\label{prop:simultaneous}
There is an open dense full-measure semialgebraic set
$\Ggen\subseteq\cP_n$ such that every $f\in\Ggen$ has a unique Boolean
minimizer and, at every order $d_\star(f)\le d\le n$, the relaxation admits
a rank-$(N_d-1)$ optimal certificate.
\end{proposition}
\begin{proof}
First exclude the finite union of proper tie hyperplanes
$E_{\mathrm{tie}}:=\bigcup_{a\ne b}\{f:f(a)=f(b)\}$; each is proper because
some linear monomial distinguishes $a$ from $b$. For each pair $(a,d)$, let
$B_{a,d}:=\cE(a,d)\setminus\cE_{\mathrm{sc}}(a,d)$ and lift it to the full
coefficient space by the additive constant,
$\widetilde B_{a,d}:=\{g+c:g\in B_{a,d},\ c\in\R\}$. The map
$(g,c)\mapsto g+c$ is injective on this domain (because $g(a)=0$ recovers
$c=f(a)$), so for nonempty $B_{a,d}$,
\[
 \dim\widetilde B_{a,d}=\dim B_{a,d}+1<|\cI_{2d}|\le2^n
\]
by \cref{thm:cell-generic}. Let $E_{\mathrm{bad}}$ be the union of
$E_{\mathrm{tie}}$ with the finitely many $\widetilde B_{a,d}$, and set
$\Ggen:=\cP_n\setminus\cl E_{\mathrm{bad}}$. Finite unions and closures
preserve semialgebraicity and dimension, so $\dim\cl E_{\mathrm{bad}}<2^n$
and $\Ggen$ is open, dense, and full measure. For $f\in\Ggen$ with
minimizer $x_\star$, \cref{prop:eventual} gives $d_\star(f)\le n$ and
unique rank-one exactness at every order $d\in[d_\star(f),n]$. At each
such order, $f-f(x_\star)\in\cE(x_\star,d)$. Since $f$ avoids the
corresponding lifted bad set, this normalized objective lies in
$\cE_{\mathrm{sc}}(x_\star,d)$.
\end{proof}

The proof of \cref{thm:main} in \cref{sec:genericity} assembles
\cref{prop:simultaneous} with
\cref{prop:encoding,lem:slack-trace,prop:boolean-modulus,thm:rigidity} and
the transfer results of \cref{sec:algorithms}.

%% file: sections/app_experiments.tex
\section{Details of the Illustrative Experiments}
\label{app:experiments}

This appendix records the setup behind \cref{fig:motivation} and one
further illustration for a higher-degree objective. The experiments are
illustrations of the theory rather than a benchmark; no timing comparison
is intended.

\paragraph{General setup.}
All runs except the saved run of panel (a) use the exact ADMM of
\cref{sec:algorithms} with multiplier step $\tau=1$ and a fixed penalty
$\sigma$ (reported per experiment), implemented in MATLAB, stopped when
the maximal KKT residual falls below $10^{-10}$. The maximal KKT residual
is $\max\{r_{\mathrm p},r_{\mathrm d},r_{\mathrm g}\}$ with
$r_{\mathrm p}=\normtwo{\cA X-b}/(1+\normtwo{b})$,
$r_{\mathrm d}=\normF{\cA^*\lambda+S-C}/(1+\normF{C})$, and
$r_{\mathrm g}=|\ip{C}{X}-b^\top\lambda|/(1+|\ip{C}{X}|+|b^\top\lambda|)$.
Symmetric matrices
are vectorized by the orthonormal $\operatorname{svec}$, with off-diagonal
entries scaled by $\sqrt2$, so Euclidean distances in these coordinates
equal Frobenius distances. Panels (a) and (c) of each figure report the
fixed-point residual $\normF{Z^{k+1}-Z^k}$ of the shadow sequence of
\cref{sec:algorithms} on a logarithmic scale. Unique rank-one exactness
and the existence of a strictly complementary certificate are verified
once per instance. We solve both sides of the relaxation with an
interior-point solver at tolerance $10^{-14}$ and inspect the spectra of
the optimal pair. Throughout, ``certified'' means that the Boolean minimum
and its uniqueness are established by exhaustive enumeration of all
$2^{n}$ cube points (exact up to floating-point summation); rank and
strict-complementarity statements are numerical, at the stated
tolerances. All rank readouts in this appendix use a $10^{-6}$ relative
eigenvalue tolerance. All new runs were performed on a laptop with an
Apple M2 Max processor and $64$\,GB of RAM.


\paragraph{Instances.}
\emph{BQP-20}: the binary quadratic instance BQP-r1-20-1 of
\cite{kang25arxiv-admmsdp-linearconv}, with $n=20$ and standard Gaussian
$Q$ and $c$; \cref{fig:motivation}(a) is that paper's saved run on the
standard (non-quotient) second-order relaxation, with $N=231$,
$m=20{,}601$, and $\sigma=100$. \emph{BQP-20q}: an independent Gaussian
instance with $n=20$, used for \cref{fig:motivation}(b)--(c) on the
quotient relaxation of \cref{sec:degeneracy} ($N=211$, $m=16{,}171$,
$\sigma=1$). It is certified by direct enumeration of the $2^{20}$ cube
points and an interior-point solve. The minimizer is unique, with the
second-smallest cube value $0.69$ above the minimum; the relaxation value
matches the minimum to relative error below $10^{-13}$; and a strictly
complementary certificate of rank $N-1=210$ exists. \emph{PBO4-20}: a
dense degree-$\le4$ Boolean objective with standard Gaussian coefficients
and $n=20$, relaxed at order $d=2$ ($N=211$, $m=16{,}171$, $\sigma=1$). It
is certified like BQP-20q, with the cube scan performed by a fast
Walsh--Hadamard transform. The minimizer is unique, with the
second-smallest cube value $15.6$ above the minimum; the relaxation value
matches the minimum to relative error below $10^{-13}$; and a strictly
complementary certificate of rank $210$ exists. Not
every Gaussian quartic is order-two exact: the first sampled seed has
relative relaxation gap $7.8\times10^{-2}$, and the certified instance is
the second seed.

\paragraph{Cold start (\cref{fig:motivation}(a)).}
The run converges in $4{,}583$ iterations. The limiting primal iterate has
numerical rank one, identified at iteration $1{,}113$ (the dotted line in
\cref{fig:motivation}(a)). The limiting certificate has numerical rank
$210$ out of $N-1=230$, with
$\lambda_{\min}(X_{\mathrm{lim}}+\sigma S_{\mathrm{lim}})\approx
4\times10^{-15}$. So the limiting pair is not strictly complementary. The
sharp linear phase begins essentially at rank identification.

\paragraph{Sampling singular endpoints (\cref{fig:motivation}(b)).}
For fixed $\sigma$, the fixed points of the shadow recursion are
$\Zopt=\{X_\star-\sigma S:S\in\Sopt\}$ (\cref{sec:rigidity}), so sampling
an endpoint means sampling an optimal certificate. We draw a symmetrized
standard Gaussian matrix $R$ and minimize $\ip{R}{S}$ over the optimal
certificate slice $\Sopt=(\Ssc+\cL)\cap\Sym_+(V)$ with an interior-point
solver. For a generic $R$ the minimizer is unique and is an extreme point
of $\Sopt$. Extreme points of the positive-dimensional set $\Sopt$ lie on
its relative boundary, where certificates lose rank. One hundred
independent draws on BQP-20q give certificates $S_i$ of numerical rank
$68$ to $72$ (median $69$), against the strictly complementary rank
$N-1=210$. The corresponding endpoints
$\bar{Z}_i=X_\star-\sigma S_i$ have numerical ranks $68$ to $72$ (median
$70$), far below the full rank $N=211$ attained at the strictly
complementary endpoint; \cref{fig:motivation}(b) plots their
distribution. Both readouts are relative to the matrix's own largest
eigenvalue; since $\normop{\bar{Z}_i}\approx N$ while $\normop{S_i}$ is
much smaller, the numerical rank of $\bar{Z}_i$ need not equal
$1+\rank(S_i)$. The endpoints
are also essentially different from one another: the pairwise distances
$\normF{\bar{Z}_i-\bar{Z}_j}=\sigma\normF{S_i-S_j}$ lie in $[15.7,16.7]$,
comparable to the certificate norms $\normF{S_i}\approx 13$ themselves.

\paragraph{Near-singular warm starts (\cref{fig:motivation}(c)).}
The $i$-th run starts at $Z^0_i=\bar{Z}_i+\eps G_i$ with $\eps=10^{-3}$,
where $G_i$ has independent standard Gaussian entries in the
$\operatorname{svec}$ coordinates. The perturbation has norm about
$\eps\sqrt{N(N+1)/2}\approx 0.15$, below $10^{-3}\,\normF{\bar{Z}_i}$. So
every start lies near a singular endpoint rather than at it. All $100$
runs converge. The iteration counts lie in $[4{,}008,\,5{,}124]$ (median
$4{,}922$). Each
limiting moment iterate has rank one, and each limiting certificate stays
within $0.11$ in Frobenius norm (about one percent of its size) of its
sampled certificate $S_i$. At the stated tolerance these limiting
certificates have numerical rank $210$, so the iteration stops at
nonsingular endpoints close to the singular samples.

\paragraph{A higher-degree objective.}
\Cref{fig:pbo4} repeats the full protocol of \cref{fig:motivation} on
PBO4-20. From the all-zeros start (panel (a)), the residual enters the
same sharp linear phase after a long transient. The run converges in $50{,}014$ iterations; on this instance
the cold start happens to select a strictly complementary endpoint.
Panels (b)--(c) report one hundred near-singular warm starts, generated
exactly as for BQP-20q. The sampled certificates have numerical ranks $68$
to $80$ (median $72$). The endpoints $\bar{Z}_i$ have numerical ranks $69$
to $71$, with pairwise distances in $[8.6,10.0]$ (the certificate norms
are about $57$). All $100$ runs converge, in $5{,}420$ to $5{,}497$
iterations, and the limiting certificates again have numerical rank $210$. The
phenomenon, and our theory, are not special to quadratic objectives.

\begin{figure}[t]
\centering
\begin{minipage}[t]{0.32\textwidth}\centering
 \includegraphics[width=\linewidth]{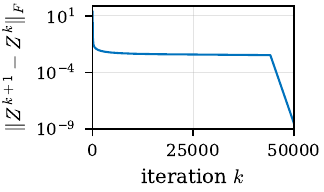}\\[1pt]
 {\small (a) cold start, SC limit}
\end{minipage}\hfill
\begin{minipage}[t]{0.32\textwidth}\centering
 \includegraphics[width=\linewidth]{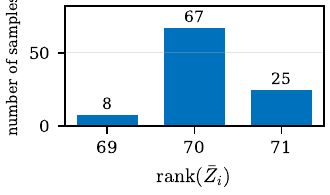}\\[1pt]
 {\small (b) ranks of the endpoints $\bar{Z}_i$}
\end{minipage}\hfill
\begin{minipage}[t]{0.32\textwidth}\centering
 \includegraphics[width=\linewidth]{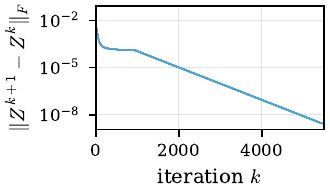}\\[1pt]
 {\small (c) $100$ near-singular warm starts}
\end{minipage}
\caption{The protocol of \cref{fig:motivation} on the certified dense
quartic instance PBO4-20 ($N=211$, order-two quotient relaxation). (a)
All-zeros start: $50{,}014$ iterations; here the limiting pair is strictly
complementary. (b) Numerical ranks of $100$
sampled singular endpoints $\bar{Z}_i$, far below $N=211$. (c) The $100$
runs started near them all converge locally linearly.}
\label{fig:pbo4}
\end{figure}

%% file: sections/ai_disclosure.tex
\section*{AI Methodologies}
\label{sec:ai}
The mathematical content of this paper---its definitions, theorems, and
proofs---was developed by the author. Large language models (OpenAI ChatGPT
and Anthropic Claude) were used to organize the manuscript, to search and
verify literature, to review drafts adversarially, and to copy-edit and
compress the exposition throughout; this assistance materially affected the
structure of the introduction and related-work discussion and the
presentation of the proofs. Every mathematical statement was verified by the
author, who takes full responsibility for this paper.